\documentclass[12pt]{amsart}
\usepackage{amsfonts, amssymb, mathrsfs}
\usepackage{fullpage, verbatim}
\usepackage[colorlinks=true]{hyperref}
\usepackage{breakurl}

\usepackage{etoolbox}
\makeatletter
\let\ams@starttoc\@starttoc
\makeatother
\usepackage[parfill]{parskip}
\makeatletter
\let\@starttoc\ams@starttoc
\patchcmd{\@starttoc}{\makeatletter}{\makeatletter\parskip\z@}{}{}
\makeatother

\newcommand\CA{{\mathscr A}} 
\newcommand\CB{{\mathscr B}}
\newcommand\CC{{\mathscr C}}

\newcommand\CIF{{\mathcal {IF}}} 
\newcommand\CIFM{{\mathcal {IFM}}}

\newcommand\BBK{{\mathbb K}}

\newcommand\BBZ{{\mathbb Z}}

\newcommand {\GAP}{\textsf{GAP}}  
  
\newcommand {\Singular}{\textsf{SINGULAR}}  
\newcommand {\Sage}{\textsf{SAGE}}  

\newcommand\Ann{{\operatorname{Ann}}}
\newcommand\codim{\operatorname{codim}}

\newcommand\Der{{\operatorname{Der}}}

\newcommand\Fix{{\operatorname{Fix}}}

\newcommand\GL{\operatorname{GL}}

\newcommand\pdeg{\operatorname{pdeg}}

\renewcommand\th{{^{\text{th}}}}

\newcommand{\one}{{1\!\!1}}

\numberwithin{equation}{section}

\theoremstyle{plain}
\newtheorem{lemma}[equation]{Lemma}
\newtheorem{theorem}[equation]{Theorem}

\newtheorem{question}[equation]{Question}
\newtheorem{corollary}[equation]{Corollary}
\newtheorem{proposition}[equation]{Proposition}
\theoremstyle{definition}
\newtheorem{defn}[equation]{Definition}
\newtheorem{remark}[equation]{Remark}
\newtheorem{example}[equation]{Example}


\begin{document}

\title[ Inductive Freeness of Multiarrangements of Reflection Groups ]
{Inductive Freeness of 
Ziegler's Canonical Multiderivations for Reflection Arrangements}

\author[T. Hoge]{Torsten Hoge}
\address
{Institut f\"ur Algebra, Zahlentheorie und Diskrete Mathematik,
Fakult\"at f\"ur Mathematik und Physik,
Leibniz Universit\"at Hannover,
Welfengarten 1,
30167 Hannover, Germany}
\email{hoge@math.uni-hannover.de}

\author[G. R\"ohrle]{Gerhard R\"ohrle}
\address
{Fakult\"at f\"ur Mathematik,
Ruhr-Universit\"at Bochum,
D-44780 Bochum, Germany}
\email{gerhard.roehrle@rub.de}



\allowdisplaybreaks

\begin{abstract}
Let $\CA$ be a free hyperplane arrangement.
In 1989, Ziegler showed that 
the restriction $\CA''$ of $\CA$
to any hyperplane 
endowed with the natural multiplicity
is then a free multiarrangement.
We initiate
a study of the stronger freeness 
property of inductive freeness for these canonical free
multiarrangements and investigate them
for the underlying class
of reflection arrangements.

More precisely, let $\CA = \CA(W)$ 
be the reflection arrangement
of a complex reflection group $W$.
By work of Terao, each such reflection 
arrangement is free.
Thus so is 
Ziegler's canonical multiplicity
on the restriction  $\CA''$
of $\CA$ to a hyperplane.
We show that the latter 
is inductively free as a multiarrangement
if and only if $\CA''$ itself is
inductively free.
\end{abstract}

\maketitle

\setcounter{tocdepth}{1}
\tableofcontents


\section{Introduction}

The class of free arrangements, respectively free multiarrangments, 
plays a fundamental role in the theory of 
hyperplane arrangements, respectively 
multiarrangements. 
In his seminal work \cite{ziegler:multiarrangements}, Ziegler 
introduced the notion of multiarrangements and initiated the study of their 
freeness.  
We begin by recalling Ziegler's 
fundamental construction from \emph{loc.~cit.} 

\begin{defn}
\label{def:kappa}
Let $\CA$ be an arrangement.
Fix $H_0 \in \CA$ and  consider the restriction 
$\CA''$ with respect to $H_0$.
Define the \emph{canonical multiplicity} 
$\kappa$ on $\CA''$ as follows. For $Y \in \CA''$ set 
\[
\kappa(Y) := |\CA_Y| -1,
\]
i.e., $\kappa(Y)$ is the number of hyperplanes in $\CA \setminus\{H_0\}$
lying above $Y$.
Ziegler showed that freeness of $\CA$ implies 
freeness of $(\CA'', \kappa)$ as follows.
\end{defn}

\begin{theorem}
[{\cite[Thm.~11]{ziegler:multiarrangements}}]
\label{thm:zieglermulti}
Let $\CA$ be a free arrangement with exponents
$\exp \CA = \{1, e_2, \ldots, e_\ell\}$.
Let $H_0 \in \CA$ and consider the restriction 
$\CA''$ with respect to $H_0$.
Then the multiarrangement $(\CA'', \kappa)$ is free with
exponents
$\exp (\CA'', \kappa) = \{e_2, \ldots, e_\ell\}$. 
\end{theorem}

Note that 
the converse of Theorem \ref{thm:zieglermulti} is false.
For let $\CA$ be a non-free $3$-arrangement, 
cf.~\cite[Ex.~4.34]{orlikterao:arrangements}.
Since $\CA''$ is of rank $2$,
$(\CA'', \kappa)$ is free, 
\cite[Cor.~7]{ziegler:multiarrangements}.
Nevertheless, 
Ziegler's construction 
and in particular the question of a converse
of Theorem \ref{thm:zieglermulti}
under suitable additional hypotheses 
play an important role in the 
study of free simple arrangements, 
\cite[Thms.~2.1, 2.2]{yoshinaga:free04},
\cite{yoshinaga:free05},
\cite[Cor.~4.2]{abeyoshinaga},
\cite[Thm.~2]{schulze:free} and
\cite[Cor.~1.35]{yoshinaga:free14}.

Because of the relevance of Ziegler's multiplicity
in the theory of free arrangements, 
it is natural to investigate stronger
freeness properties for $(\CA'', \kappa)$ and
specifically to ask for an analogue of
Theorem \ref{thm:zieglermulti} for 
inductive freeness.

\begin{question}
\label{qu:kappa}
Is it the case that
$(\CA'', \kappa)$ is inductively free
whenever $\CA$ is?
\end{question}

In this paper we initialize the study 
of this question and examine the multiarrangements $(\CA'', \kappa)$ where 
the underlying
class consists of reflection arrangements.
It turns out that for this class the answer is ``yes''.
While inductive freeness is a combinatorial 
property for simple arrangements, i.e.~it
only depends on the underlying intersection lattice, in contrast,
 for multiarrangements it is not combinatorial
in general, \cite[Prop.~10]{ziegler:multiarrangements}.
So one might not expect the implication in Question \ref{qu:kappa}
to hold for general arrangements $\CA$. However, having  
extensively investigated many examples, 
we have been unable to come across a counterexample. 
So it is quite likely that this implication does indeed always hold.

It can be quite cumbersome to decide whether a given 
(multi)arrangement is inductively free, see
for instance 
\cite[\S 5.2]{cuntz:indfree},
\cite[Lem.~4.2]{amendhogeroehrle:indfree},
\cite[Lem.~3.5]{hogeroehrle:indfree}, 
\cite[Thm.~1.1]{conradroehrle:indfree}
and  
\cite[Thm.~1.4]{cuntzroehrleschauenburg:ideal}.  
This might entail an extensive perusal of
a large number of chains of free subarrangements.

If $\CA = \CA(W)$ is a reflection arrangement of 
the complex reflection group $W$, 
then $\CA$ is free, thanks to work of 
Terao \cite{terao:freereflections}.
Thus $(\CA'', \kappa)$ is also free, by 
Theorem \ref{thm:zieglermulti}.

In our main result, we
classify all instances when 
Ziegler's canonical multiplicity 
$(\CA'', \kappa)$ 
is inductively free in case 
the underlying arrangement $\CA$
is a reflection arrangement.
In particular, we answer 
Question \ref{qu:kappa} in the affirmative
for this class of arrangements.

Because of the compatibility of
the product constructions for  
inductive freeness for simple arrangements,
\cite[Prop.~2.10]{hogeroehrle:indfree},
as well as for multiarrangements,
\cite[Thm.~1.4]{hogeroehrleschauenburg:free}
(cf.~Theorem \ref{thm:products}),
the question of inductive freeness 
of $(\CA'', \kappa)$ readily reduces to 
the case when $\CA$ is irreducible.
Thus for a reflection arrangement 
$\CA(W)$, we may assume that 
$W$ is irreducible.

\begin{theorem}
\label{thm:kappa1}
Let $\CA = \CA(W)$ be the reflection arrangement of 
the irreducible complex reflection group $W$.
Let $\CA''$ be the restriction of $\CA$ to 
a hyperplane in $\CA$.
Then $(\CA'', \kappa)$ is inductively free
if and only if one of the following holds:
\begin{itemize}
\item[(i)] $\CA$ is inductively free; or 
\item[(ii)] $\CA$ is non-inductively free of rank at most $4$. 
\end{itemize}
\end{theorem}

We record several consequences of Theorem \ref{thm:kappa1}.
Our first observation follows from the theorem along with  
the classification of all 
inductively free restrictions of 
reflection arrangements from 
\cite[Thm.~1.2]{amendhogeroehrle:indfree},
see Theorem \ref{thm:indfree2} below.

\begin{corollary}
\label{cor:kappa2}
Let $\CA = \CA(W)$ be the reflection arrangement of 
the complex reflection group $W$.
Then $(\CA'', \kappa)$ is inductively free
if and only if $\CA''$ itself is inductively free.
\end{corollary}

We mention that Corollary \ref{cor:kappa2} is false in general.
For the failure of the forward implication, see
\cite[Ex.~ 2.16]{hogeroehrle:indfree}. 
There we present an example of an inductively free
arrangement $\CA$ where a particular restriction 
to a hyperplane $\CA''$ fails to be free.
In particular, $\CA''$ is not inductively free.
Nevertheless, since $\CA$ is free, 
so is $(\CA'', \kappa)$, thanks to 
Theorem \ref{thm:zieglermulti}.
One can check that in this instance 
$(\CA'', \kappa)$ is actually inductively free.
Nevertheless, this example is consistent with the assertion   
in Question \ref{qu:kappa}.
Example  \ref{ex:kappa2} shows that 
the reverse implication in Corollary \ref{cor:kappa2}
also fails in general.
In view of these elementary examples, 
the equivalence of Corollary \ref{cor:kappa2}
is rather striking and underlines the very special
nature of reflection arrangements.

In addition to the canonical multiplicity from 
Definition \ref{def:kappa}, we study 
multiplicities which are
concentrated at a single hyperplane. 
These were introduced by
Abe, Terao and Wakefield,
{\cite[\S 5]{abeteraowakefield:euler}.
While in general, $(\CA, \mu)$ need not be free 
for a free hyperplane arrangement $\CA$ and 
an arbitrary multiplicity $\mu$, 
e.g.~see \cite[Ex.~14]{ziegler:multiarrangements},
for these concentrated multiplicities freeness is 
also inherited from the simple arrangement.
It turns out that they are closely related to Ziegler's
canonical multiplicity, 
see Propositions \ref{prop:delta} and 
\ref{prop:reversedelta}.

\begin{defn}
\label{def:delta}
Let $\CA$ be a simple arrangement.
Fix $H_0 \in \CA$ and $m_0 \in \BBZ_{\ge 1}$ and 
define the 
\emph{multiplicity $\delta$ concentrated at $H_0$}
by
\[
\delta(H) := \delta_{H_0,m_0}(H) := 
\begin{cases}
m_0 & \text{ if } H = H_0,\\
1   & \text{ else}.
\end{cases}
\]
\end{defn}

It turns out that in this instance, both
$\CA$ and $(\CA, \delta)$ 
inherit freeness from each other,
as opposed to the general case
indicated above.
We record this in our next result.

\begin{theorem}
\label{thm:delta}
Let $\CA$ be an arrangement.
Fix $H_0 \in \CA$, $m_0 \in \BBZ_{\ge 1}$ and let 
$\delta = \delta_{H_0,m_0}$ be 
the multiplicity concentrated at $H_0$, 
as in Definition \ref{def:delta}.
Then 
$\CA$ is free 
with exponents
$\exp \CA = \{1, e_2, \ldots, e_\ell\}$
if and only if
$(\CA, \delta)$ is free with exponents
$\exp (\CA, \delta) = \{m_0, e_2, \ldots, e_\ell\}$. 
\end{theorem}

The forward implication of 
Theorem \ref{thm:delta} is 
\cite[Prop.~5.2]{abeteraowakefield:euler}
and the reverse implication is
Proposition \ref{prop:reversedelta}(i).

As in the case of Ziegler's canonical multiplicity, one
might also speculate whether 
the analogue for inductive freeness holds for
the forward implication of 
Theorem \ref{thm:delta}.

\begin{question}
\label{qu:delta}
Is it the case that 
$(\CA, \delta)$ is inductively free
whenever $\CA$ is?
\end{question}

It turns out that 
Questions \ref{qu:kappa} 
and \ref{qu:delta} are closely related:
the assertion of 
Question \ref{qu:delta} 
follows from that in 
Question \ref{qu:kappa},
see Corollary \ref{cor:indfreedelta}.

Let $\CA = \CA(W)$ be the reflection arrangement of 
the complex reflection group $W$.
Since $\CA$ is free, 
it follows from 
Proposition \ref{prop:delta}(i) 
that the multiarrangement
$(\CA, \delta)$ is also free.
In \cite[Thm.~1.1]{hogeroehrle:indfree},
we classified all inductively free 
reflection arrangements.
In our second main result we 
extend this classification to 
multiarrangements of the form
$(\CA(W), \delta)$. 
This in particular gives an affirmative 
answer to Question \ref{qu:delta}
for the class of reflection arrangements.

\begin{theorem}
\label{thm:Wdelta}
Let $\CA = \CA(W)$ be the reflection arrangement of 
the complex reflection group $W$.
Then $(\CA, \delta)$ is inductively free
if and only if $\CA$ is inductively free.
\end{theorem}

The paper is organized as follows.
In Section \ref{ssect:free} we recall 
the fundamental results for free arrangements, 
in particular Terao's 
Addition Deletion Theorem \ref{thm:add-del-simple}
and the subsequent
notion of an inductively free arrangement.
The classification of the inductively free
reflection arrangements from \cite[Thm.~1.1]{hogeroehrle:indfree}, 
and the 
classification of the inductively free
restrictions of reflection arrangements 
from \cite[Thm.~1.2]{amendhogeroehrle:indfree}
are restated in Section \ref{ssect:refl}.
Section \ref{ssec:freemulti} is devoted to 
multiarrangements and their freeness. 
Here we present the Addition Deletion theorem 
due to Abe, Terao, and Wakefield, 
\cite[Thm.~0.8]{abeteraowakefield:euler}.
In Section \ref{subsec:ziegler} we treat
Ziegler's canonical and 
concentrated multiplicities.
This is followed by a 
discussion of inductive freeness 
for multiarrangements in Section \ref{ssec:indutive}.
Here we recall results from 
\cite{hogeroehrleschauenburg:free}
which show the compatibility of this notion with 
products and localization for multiarrangements
that are used in the sequel.

Section \ref{sec:filtrations} is central in the paper.
Here we prove 
some general results for free arrangements which show that under 
weak conditions on the exponents of 
a free arrangement $\CA$ and
a given free restriction $\CA^H$, \emph{every} 
multiplicity between the simple and 
the Ziegler multiplicity on $\CA^H$
is  already free, see Corollary \ref{cor:multifreetube}.
This is then applied to reflection arrangements in 
Proposition \ref{prop:multifreetube}.
We then apply 
these methods to the imprimitive reflection groups $G(r,1,\ell-1)$.
Combined with the fact that 
the reflection arrangement $\CA(G(r,1,\ell-1))$ 
itself is inductively free
(cf.~Theorem \ref{thm:indfree1}), 
we show  that $(\CA(G(r,1,\ell))'',\kappa)$ is inductively free
in Proposition \ref{prop:monomial}. 
As a further relevant consequence
for our classification, we derive that
$(\CA(W)'',\kappa)$ is inductively free,
for $W$ of type $F_4$, $H_4$, $G_{31}$ and $G_{32}$ in  
Corollary \ref{cor:indfreerank3}.

In Section \ref{sec:proofs} we present 
our proofs of Theorems \ref{thm:kappa1} and \ref{thm:Wdelta}.
The fact that inductively free multiarrangements 
are closed under localization proves extremely useful in this context, 
see Theorem \ref{thm:localmulti}.
By this means, we can for instance deduce that 
$(\CA(E_6)'',\kappa)$ and $(\CA(E_7)'',\kappa)$ are inductively free
from the fact that $(\CA(E_8)'',\kappa)$ is so, see 
Lemma \ref{lem:coxeterexceptional}. 
The same theorem in turn allows us to derive 
that $(\CA(G_{34})'',\kappa)$ is not inductively free 
from the fact that $(\CA(G(3,3,5))'',\kappa)$ isn't either, see 
Lemma \ref{lem:g33}.

Due to its large rank,  
the proof of inductive freeness for $(\CA(E_8)'',\kappa)$
is the most challenging among all groups of exceptional 
type. It involves delicate arguments of induction, see
Lemma \ref{lem:coxeterexceptional}. 
A crucial ingredient in our proof is the fact 
that on $\CA(E_8)''$
\emph{every} multiplicity between the simple one and Ziegler's multiplicity 
$\kappa$ is free, thus readily providing free filtrations of  
$(\CA(E_8)'',\kappa)$ for our analysis, see
Proposition \ref{prop:multifreetube}.

\section{Recollections and Preliminaries}
\label{sect:prelim}

\subsection{Hyperplane Arrangements}
\label{ssect:hyper}
Let $V = \BBK^\ell$ 
be an $\ell$-dimensional $\BBK$-vector space.
A \emph{hyperplane arrangement} is a pair
$(\CA, V)$, where $\CA$ is a finite collection of hyperplanes in $V$.
Usually, we simply write $\CA$ in place of $(\CA, V)$.
We write $|\CA|$ for the number of hyperplanes in $\CA$.
The empty arrangement in $V$ is denoted by $\Phi_\ell$.

The \emph{lattice} $L(\CA)$ of $\CA$ is the set of subspaces of $V$ of
the form $H_1\cap \dotsm \cap H_i$ where $\{ H_1, \ldots, H_i\}$ is a subset
of $\CA$. 
For $X \in L(\CA)$, we have two associated arrangements, 
firstly
$\CA_X :=\{H \in \CA \mid X \subseteq H\} \subseteq \CA$,
the \emph{localization of $\CA$ at $X$}, 
and secondly, 
the \emph{restriction of $\CA$ to $X$}, $(\CA^X,X)$, where 
$\CA^X := \{ X \cap H \mid H \in \CA \setminus \CA_X\}$.
Note that $V$ belongs to $L(\CA)$
as the intersection of the empty 
collection of hyperplanes and $\CA^V = \CA$. 
The lattice $L(\CA)$ is a partially ordered set by reverse inclusion:
$X \le Y$ provided $Y \subseteq X$ for $X,Y \in L(\CA)$.

If $0 \in H$ for each $H$ in $\CA$, then 
$\CA$ is called \emph{central}.
If $\CA$ is central, then the \emph{center} 
$T_\CA := \cap_{H \in \CA} H$ of $\CA$ is the unique
maximal element in $L(\CA)$  with respect
to the partial order.
We have a \emph{rank} function on $L(\CA)$: $r(X) := \codim_V(X)$.
The \emph{rank} $r := r(\CA)$ of $\CA$ 
is the rank of a maximal element in $L(\CA)$.
The $\ell$-arrangement $\CA$ is \emph{essential} 
provided $r(\CA) = \ell$.
If $\CA$ is central and essential, then $T_\CA =\{0\}$.
Throughout, we only consider central arrangements.

\subsection{Free Hyperplane Arrangements}
\label{ssect:free}
Let $S = S(V^*)$ be the symmetric algebra of the dual space $V^*$ of $V$.
If $x_1, \ldots , x_\ell$ is a basis of $V^*$, then we identify $S$ with 
the polynomial ring $\BBK[x_1, \ldots , x_\ell]$.
Letting $S_p$ denote the $\BBK$-subspace of $S$
consisting of the homogeneous polynomials of degree $p$ (along with $0$),
$S$ is naturally $\BBZ$-graded: $S = \oplus_{p \in \BBZ}S_p$, where
$S_p = 0$ in case $p < 0$.

Let $\Der(S)$ be the $S$-module of algebraic $\BBK$-derivations of $S$.
Using the $\BBZ$-grading on $S$, $\Der(S)$ becomes a graded $S$-module.
For $i = 1, \ldots, \ell$, 
let $D_i := \partial/\partial x_i$ be the usual derivation of $S$.
Then $D_1, \ldots, D_\ell$ is an $S$-basis of $\Der(S)$.
We say that $\theta \in \Der(S)$ is 
\emph{homogeneous of polynomial degree p}
provided 
$\theta = \sum_{i=1}^\ell f_i D_i$, 
where $f_i$ is either $0$ or homogeneous of degree $p$
for each $1 \le i \le \ell$.
In this case we write $\pdeg \theta = p$.

Let $\CA$ be an arrangement in $V$. 
Then for $H \in \CA$ we fix $\alpha_H \in V^*$ with
$H = \ker(\alpha_H)$.
The \emph{defining polynomial} $Q(\CA)$ of $\CA$ is given by 
$Q(\CA) := \prod_{H \in \CA} \alpha_H \in S$.

The \emph{module of $\CA$-derivations} of $\CA$ is 
defined by 
\[
D(\CA) := \{\theta \in \Der(S) \mid \theta(\alpha_H) \in \alpha_H S
\text{ for each } H \in \CA \} .
\]
We say that $\CA$ is \emph{free} if the module of $\CA$-derivations
$D(\CA)$ is a free $S$-module.

With the $\BBZ$-grading of $\Der(S)$, 
also $D(\CA)$ 
becomes a graded $S$-module,
\cite[Prop.~4.10]{orlikterao:arrangements}.
If $\CA$ is a free arrangement, then the $S$-module 
$D(\CA)$ admits a basis of $\ell$ homogeneous derivations, 
say $\theta_1, \ldots, \theta_\ell$, \cite[Prop.~4.18]{orlikterao:arrangements}.
While the $\theta_i$'s are not unique, their polynomial 
degrees $\pdeg \theta_i$ 
are unique (up to ordering). This multiset is the set of 
\emph{exponents} of the free arrangement $\CA$
and is denoted by $\exp \CA$.

The fundamental \emph{Addition Deletion Theorem} 
due to Terao  \cite{terao:freeI} plays a 
crucial role in the study of free arrangements, 
\cite[Thm.~4.51]{orlikterao:arrangements}.

\begin{theorem}
\label{thm:add-del-simple}
Suppose $\CA \neq \Phi_\ell$ and
let $(\CA, \CA', \CA'')$ be a triple of arrangements. Then any 
two of the following statements imply the third:
\begin{itemize}
\item[(i)] $\CA$ is free with $\exp\CA = \{ b_1, \ldots , b_{\ell -1}, b_\ell\}$;
\item[(ii)] $\CA'$ is free with $\exp\CA' = \{ b_1, \ldots , b_{\ell -1}, b_\ell-1\}$;
\item[(iii)] $\CA''$ is free with $\exp\CA'' = \{ b_1, \ldots , b_{\ell -1}\}$.
\end{itemize}
\end{theorem}

Theorem \ref{thm:add-del-simple} motivates the notion of 
\emph{inductively free} arrangements, 
cf.~\cite[Def.~4.53]{orlikterao:arrangements}:

\begin{defn}
\label{def:indfree-simple}
The class $\CIF$ of \emph{inductively free} arrangements 
is the smallest class of arrangements subject to
\begin{itemize}
\item[(i)] $\Phi_\ell \in \CIF$ for each $\ell \ge 0$;
\item[(ii)] if there exists a hyperplane $H_0 \in \CA$ such that both
$\CA'$ and $\CA''$ belong to $\CIF$, and $\exp \CA '' \subseteq \exp \CA'$, 
then $\CA$ also belongs to $\CIF$.
\end{itemize}
\end{defn}

\subsection{Reflection Arrangements}
\label{ssect:refl}
The irreducible finite complex reflection groups were 
classified by Shephard and Todd, \cite{shephardtodd}.
Let $W  \subseteq \GL(V)$ be a finite complex reflection group.
For $w \in W$, we write 
$\Fix(w) :=\{ v\in V \mid w v = v\}$ for 
the fixed point subspace of $w$.
For $U \subseteq V$ a subspace, we 
define the \emph{parabolic subgroup}
$W_U$ of $W$ by 
$W_U := \{w \in W \mid U \subseteq \Fix(w)\}$.

The \emph{reflection arrangement} $\CA = \CA(W)$ of $W$ in $V$ is 
the hyperplane arrangement 
consisting of the reflecting hyperplanes of the elements in $W$
acting as reflections on $V$.
By Steinberg's Theorem \cite[Thm.~1.5]{steinberg:invariants},
for $X \in L(\CA)$, 
the parabolic subgroup
$W_X$ is itself a complex reflection group,
generated by the unitary reflections in $W$ that are contained
in $W_X$. 
Thus, we identify the 
reflection arrangement $\CA(W_X)$
of $W_X$ as a subarrangement of $\CA$.

Using the classification and  nomenclature of Shephard and Todd \cite{shephardtodd},
we recall the main classification results
from \cite[Thm.~1.1]{hogeroehrle:indfree} and 
\cite[Thm.~1.2]{amendhogeroehrle:indfree}, respectively.

\begin{theorem}
\label{thm:indfree1}
For a finite complex reflection group  $W$,  
its reflection arrangement $\CA(W)$ is 
inductively free if and only if 
$W$ does not admit an irreducible factor
isomorphic to a monomial group 
$G(r,r,\ell)$ for $r, \ell \ge 3$, 
$G_{24}, G_{27}, G_{29}, G_{31}, G_{33}$, or $G_{34}$.
\end{theorem}

\begin{theorem}
\label{thm:indfree2}
Let $W$ be a finite, irreducible, complex 
reflection group with reflection arrangement 
$\CA = \CA(W)$ and let $X \in L(\CA)$. 
The restricted arrangement $\CA^X$ is inductively free 
if and only if one of the following holds:
\begin{itemize}
\item[(i)] 
$\CA$ is inductively free;
\item[(ii)] 
$W = G(r,r,\ell)$ and 
$\CA^X \cong \CA^k_p(r)$, where $p = \dim X$ and $p - 2 \leq k \leq p$; 
\item[(iii)] 
$W$ is one of $G_{24}, G_{27}, G_{29}, G_{31}, G_{33}$, or $G_{34}$ and $X \in L(\CA) \setminus \{V\}$ with  $\dim X \leq 3$.
\end{itemize}
\end{theorem}

\subsection{Multiarrangements}
\label{ssec:multi}
A \emph{multiarrangement}  is a pair
$(\CA, \mu)$ consisting of a hyperplane arrangement $\CA$ and a 
\emph{multiplicity} function
$\mu : \CA \to \BBZ_{\ge 0}$ associating 
to each hyperplane $H$ in $\CA$ a non-negative integer $\mu(H)$.
Alternately, the multiarrangement $(\CA, \mu)$ can also be thought of as
the multiset of hyperplanes
\[
(\CA, \mu) = \{H^{\mu(H)} \mid H \in \CA\}.
\]

The \emph{order} of the multiarrangement $(\CA, \mu)$ 
is the cardinality 
of the multiset $(\CA, \mu)$; we write 
$|\mu| := |(\CA, \mu)| = \sum_{H \in \CA} \mu(H)$.
For a multiarrangement $(\CA, \mu)$, the underlying 
arrangement $\CA$ is sometimes called the associated 
\emph{simple} arrangement, and so $(\CA, \mu)$ itself is  
simple if and only if $\mu(H) = 1$ for each $H \in \CA$. 

Let $(\CA, \mu)$ be a multiarrangement in $V$ and let 
$X \in L(\CA)$. The 
\emph{localization of $(\CA, \mu)$ at $X$} is defined to be $(\CA_X, \mu_X)$,
where $\mu_X = \mu |_{\CA_X}$.

There is a natural partial order on the 
set of multiplicities for a given simple 
arrangement, as follows.

\begin{defn}
\label{def:order}
For multiplicities $\mu_1$ and $\mu_2$ on an arrangement $\CA$, we define 
$\mu_1 \le \mu_2$ provided $\mu_1(H) \le \mu_2(H)$ for every $H$ in $\CA$.
\end{defn}

\subsection{Freeness of  multiarrangements}
\label{ssec:freemulti}
Following Ziegler \cite{ziegler:multiarrangements},
we extend the notion of freeness to multiarrangements as follows.
The \emph{defining polynomial} $Q(\CA, \mu)$ 
of the multiarrangement $(\CA, \mu)$ is given by 
\[
Q(\CA, \mu) := \prod_{H \in \CA} \alpha_H^{\mu(H)},
\] 
a polynomial of degree $|\mu|$ in $S$.

The \emph{module of $\CA$-derivations} of $(\CA, \mu)$ is 
defined by 
\[
D(\CA, \mu) := \{\theta \in \Der(S) \mid \theta(\alpha_H) \in \alpha_H^{\mu(H)} S 
\text{ for each } H \in \CA\}.
\]
We say that $(\CA, \mu)$ is \emph{free} if 
$D(\CA, \mu)$ is a free $S$-module, 
\cite[Def.~6]{ziegler:multiarrangements}.

As in the case of simple arrangements,
$D(\CA, \mu)$ is a $\BBZ$-graded $S$-module and 
thus, if $(\CA, \mu)$ is free, there is a 
homogeneous basis $\theta_1, \ldots, \theta_\ell$ of $D(\CA, \mu)$.
The multiset of the unique polynomial degrees $\pdeg \theta_i$ 
forms the set of \emph{exponents} of the free multiarrangement $(\CA, \mu)$
and is denoted by $\exp (\CA, \mu)$.

Next we recall Ziegler's analogue of Saito's criterion.

\begin{theorem}
[{\cite[Thm.~8]{ziegler:multiarrangements}}]
\label{thm:ziegler-saito}
Let $\theta_1, \dots, \theta_\ell$ be in $D(\CA, \mu)$. Then 
$\{\theta_1, \dots, \theta_\ell\}$ is a basis of $D(\CA, \mu)$ if and only if
\begin{equation}
\label{eq:ziegler-saito}
\det M(\theta_1, \dots, \theta_\ell)\ \dot=\ Q(\CA, \mu).
\end{equation}
In particular, if $\theta_1, \dots, \theta_\ell$ are all homogeneous, then 
$\{\theta_1, \dots, \theta_\ell\}$ is a basis of $D(\CA, \mu)$ if and only if
$\theta_1, \dots, \theta_\ell$ are independent over $S$ and 
\begin{equation}
\label{eq:exp}
\sum \pdeg \theta_i = \deg Q(\CA, \mu) = |\mu|. 
\end{equation}
\end{theorem}

Here the notation $\dot=$ indicates equality
up to a non-zero scalar multiple and 
$M(\theta_1, \dots, \theta_\ell)$ is the coefficient matrix of 
$\{\theta_1, \dots, \theta_\ell\}$, see
\cite[Def.~4.11]{orlikterao:arrangements}.

Note that, owing to \eqref{eq:ziegler-saito}, if $(\CA, \mu)$ is free, 
one can recover 
$\mu$ from the basis $\{\theta_1, \dots, \theta_\ell\}$.

We recall the construction from \cite{abeteraowakefield:euler} for the 
counterpart of Theorem \ref{thm:add-del-simple} in this more general setting.

\begin{defn}
\label{def:Euler}
Let $(\CA, \mu) \ne \Phi_\ell$ be a multiarrangement. Fix $H_0$ in $\CA$.
We define the \emph{deletion}  $(\CA', \mu')$ and \emph{restriction} $(\CA'', \mu^*)$
of $(\CA, \mu)$ with respect to $H_0$ as follows.
If $\mu(H_0) = 1$, then set $\CA' = \CA \setminus \{H_0\}$
and define $\mu'(H) = \mu(H)$ for all $H \in \CA'$.
If $\mu(H_0) > 1$, then set $\CA' = \CA$
and define $\mu'(H_0) = \mu(H_0)-1$ and
$\mu'(H) = \mu(H)$ for all $H \ne H_0$.

Let $\CA'' = \{ H \cap H_0 \mid H \in \CA \setminus \{H_0\}\ \}$.
The \emph{Euler multiplicity} $\mu^*$ of $\CA''$ is defined as follows.
Let $Y \in \CA''$. Since the localization $\CA_Y$ is of rank $2$, the
multiarrangement $(\CA_Y, \mu_Y)$ is free, 
\cite[Cor.~7]{ziegler:multiarrangements}. 
According to 
\cite[Prop.~2.1]{abeteraowakefield:euler},
the module of derivations 
$D(\CA_Y, \mu_Y)$ admits a particular homogeneous basis
$\{\theta_Y, \psi_Y, D_3, \ldots, D_\ell\}$,
such that $\theta_Y \notin \alpha_0 \Der(S)$
and $\psi_Y \in \alpha_0 \Der(S)$,
where $H_0 = \ker \alpha_0$.
Then on $Y$ the Euler multiplicity $\mu^*$ is defined
to be $\mu^*(Y) = \pdeg \theta_Y$.

Often, 
$(\CA, \mu), (\CA', \mu')$ and $(\CA'', \mu^*)$ 
is referred to as the \emph{triple} of 
$(\CA, \mu)$ with respect to $H_0$. 
\end{defn}

We require some core results from 
\cite{abeteraowakefield:euler}.

\begin{theorem}
[{\cite[Thm.~0.4]{abeteraowakefield:euler}}]
\label{thm:restriction}
Suppose that $(\CA, \mu) \ne \Phi_\ell$.
Fix $H_0$ in $\CA$. If both 
$(\CA, \mu)$ and $(\CA', \mu')$ are free, 
then there exists a basis 
$\{\theta_1, \ldots, \theta_\ell \}$ of 
$D(\CA', \mu')$ such that 
$\{\theta_1, \ldots, \alpha_k \theta_k, \ldots, \theta_\ell \}$
is a basis of $D(\CA, \mu)$ 
for some $1 \le k \le \ell$.
\end{theorem}

Fix $H_0 = \ker \alpha_0$ in $\CA$ and let $\CA''$ be the restriction with respect to $H_0$.
Consider the projection $S \to \overline S  := S/\alpha_0 S$, 
$f \mapsto \overline f$.
Following \cite{abeteraowakefield:euler}, there is 
a canonical map $D(\CA,\mu) \to D(\CA'', \mu^*)$,  
$\theta \mapsto \overline \theta$, where 
$\overline \theta(\overline f) := \overline {\theta(f)}$, for 
$\overline f \in \overline S$.

\begin{theorem}
[{\cite[Thm.~0.6]{abeteraowakefield:euler}}]
\label{thm:restriction2}
Suppose that $(\CA, \mu) \ne \Phi_\ell$.
Fix $H_0$ in $\CA$. 
Suppose that both 
$(\CA, \mu)$ and $(\CA', \mu')$ are free.
Let 
$\{\theta_1, \ldots, \theta_\ell \}$ be a basis of 
$D(\CA', \mu')$ 
as in Theorem \ref{thm:restriction}.
Then 
$\{\overline \theta_1, \ldots, \overline \theta_{k-1}, \overline \theta_{k+1}, \ldots, \overline \theta_\ell \}$
is a basis of $D(\CA'', \mu^*)$. 
\end{theorem}

\begin{theorem}
[{\cite[Thm.~0.8]{abeteraowakefield:euler}}
Addition Deletion Theorem for Multiarrangements]
\label{thm:add-del}
Suppose that $(\CA, \mu) \ne \Phi_\ell$.
Fix $H_0$ in $\CA$ and 
let  $(\CA, \mu), (\CA', \mu')$ and  $(\CA'', \mu^*)$ be the triple with respect to $H_0$. 
Then any  two of the following statements imply the third:
\begin{itemize}
\item[(i)] $(\CA, \mu)$ is free with $\exp (\CA, \mu) = \{ b_1, \ldots , b_{\ell -1}, b_\ell\}$;
\item[(ii)] $(\CA', \mu')$ is free with $\exp (\CA', \mu') = \{ b_1, \ldots , b_{\ell -1}, b_\ell-1\}$;
\item[(iii)] $(\CA'', \mu^*)$ is free with $\exp (\CA'', \mu^*) = \{ b_1, \ldots , b_{\ell -1}\}$.
\end{itemize}
\end{theorem}

In order to apply Theorem \ref{thm:add-del} effectively,
it is crucial to be able to determine the Euler multiplicity 
for the restrictions at hand.
The following is part of 
\cite[Prop.~4.1]{abeteraowakefield:euler}
relevant for our purposes.

\begin{proposition}
\label{prop:Euler}
Let $H_0\in\CA$,  $\CA'' = \CA^{H_0}$ and let  $X\in\CA''$.
Let $\mu$ be a multiplicity on $\CA$.
Let $\mu_0 = \mu(H_0)$. Further let $k=|\CA_X|$ and $\mu_1=\max\{\mu(H) \mid H\in\CA_X\setminus\{H_0\}\}$.
\begin{itemize}
\item[(1)] 
If $k=2$, then $\mu^*(X) = \mu_1$.
\item[(2)] 
If $k=3,\ 2\mu_0\le |\mu_X|$, and  
$2\mu_1\le |\mu_X|$, then $\mu^*(X)=\left\lfloor{|\mu_X|}/{2}\right\rfloor$.
\item[(3)] If $|\mu_X| \le 2k-1$ and $\mu_0 >0$, then $\mu^*(X) = k-1$.
\end{itemize}
\end{proposition}

\subsection{Ziegler's Multiplicity and concentrated Multiplicities} 
\label{subsec:ziegler}

Recall 
Ziegler's multiplicity $\kappa$ from Definition \ref{def:kappa}
and concentrated multiplicities 
from Definition \ref{def:delta}.
The following combines
\cite[Prop.~5.2]{abeteraowakefield:euler}, parts of its proof
and Theorem \ref{thm:zieglermulti}. 
The proof of Proposition \ref{prop:delta}(i)
given in \cite{abeteraowakefield:euler} 
depends on Theorem \ref{thm:add-del}.
In \cite[Prop.~2.17]{hogeroehrleschauenburg:free}, we
presented an elementary explicit 
construction for a homogeneous 
$S$-basis of $D(\CA, \delta)$.

\begin{proposition}
\label{prop:delta}
Let $\CA$ be a free arrangement with
exponents $\exp \CA = \{1, e_2, \ldots, e_\ell\}$.
Fix $H_0 \in \CA$, $m_0 \in \BBZ_{\ge 1}$ and let 
$\delta = \delta_{H_0,m_0}$ be 
as in Definition \ref{def:delta}.
Let  $(\CA'', \delta^*)$ be the restriction with respect to $H_0$. 
Then we have
\begin{itemize}
\item[(i)] 
$(\CA, \delta)$ is free with exponents
$\exp (\CA, \delta) = \{m_0, e_2, \ldots, e_\ell\}$;
\item[(ii)] 
$(\CA'', \delta^*) = (\CA'', \kappa)$ is free with exponents
$\exp  (\CA'', \kappa) = \{e_2, \ldots, e_\ell\}$.
\end{itemize}
\end{proposition}

Part (i) of the following is a converse 
of Proposition \ref{prop:delta}(i).

\begin{proposition}
\label{prop:reversedelta}
Let $\CA$ be an  arrangement. 
Fix $H_0 \in \CA$, $m_0 \in \BBZ_{\ge 1}$ and let 
$\delta = \delta_{H_0,m_0}$ be 
as in Definition \ref{def:delta}.
Let  $(\CA'', \delta^*)$ be the restriction with respect to $H_0$. 
Suppose that 
$(\CA, \delta)$ is free 
with exponents
$\exp (\CA, \delta) = \{m_0, e_2, \ldots, e_\ell\}$.
Then we have
\begin{itemize}
\item[(i)] 
$\CA$ is free with
exponents $\exp \CA = \{1, e_2, \ldots, e_\ell\}$.
\item[(ii)] 
$(\CA'', \delta^*) = (\CA'', \kappa)$ is free with exponents
$\exp  (\CA'', \kappa) = \{e_2, \ldots, e_\ell\}$.
\end{itemize}
\end{proposition}

\begin{proof}
Clearly, part (ii) follows from 
part (i) and Proposition \ref{prop:delta}(ii).

The proof of (i) is analogous to the one of 
\cite[Prop.~4.27]{orlikterao:arrangements}.
Let $\alpha_0 \in V^*$ with $H_0 = \ker \alpha_0$.
Now let 
\[
\Ann(H_0) = \{\theta \in D(\CA, \delta) \mid \theta(\alpha_0) = 0\}
\]
be the annihilator of $H_0$ in $D(\CA, \delta)$ which is a
graded $S$-submodule of $D(\CA, \delta)$.
Let $\theta_E$ be the Euler derivation in $\Der(S)$, 
\cite[Def.~4.7]{orlikterao:arrangements}.
Define
\[
\theta_\delta := \alpha_0^{m_0-1} \theta_E.
\]
Then $\theta_\delta$ belongs to $D(\CA, \delta)$.
Moreover for each $\theta \in D(\CA, \delta)$, we have 
\[
\theta - \frac{\theta(\alpha_0)}{\alpha_0^{m_0}} \theta_\delta
\in \Ann(H_0).
\]
Therefore, since 
\[
S \theta_\delta \cap \Ann(H_0) = \{ 0 \},
\]
we derive that 
\[
D(\CA, \delta) = S \theta_\delta  \oplus \Ann(H_0)
\]
is a direct sum of $S$-modules.
Let $\{\theta_2, \ldots, \theta_\ell\}$ be a
minimal system of homogeneous generators of the 
$S$-module $\Ann(H_0)$.
Then $\{\theta_\delta, \theta_2, \ldots, \theta_\ell\}$ 
is a 
minimal system of homogeneous generators of 
$D(\CA, \delta)$.
It follows from 
\cite[Thm.~A.19]{orlikterao:arrangements}, that 
$\{\theta_\delta, \theta_2, \ldots, \theta_\ell\}$ 
is a homogeneous $S$-basis of 
$D(\CA, \delta)$.

Since $\Ann(H_0) \subset D(\CA, \delta) \subset D(\CA)$ are
$S$-submodules, 
$\{\theta_2, \ldots, \theta_\ell\}$ is linearly independent 
over $S$ in $D(\CA)$. 
Since $|\delta| = |\CA| + m_0 -1$ 
and $\pdeg \theta_\delta = m_0$,
it follows from 
Ziegler's analogue of Saito's criterion 
\eqref{eq:exp}
that 
$\sum_{i = 2}^\ell \pdeg \theta_i = |\CA| -1$.
Clearly, $\theta_E$ belongs to $D(\CA)$ but not to $\Ann(H_0)$.
Consequently, 
it follows from Saito's criterion,  
\cite[Thm.~4.23]{orlikterao:arrangements},
that $\CA$ is free with 
$\{\theta_E, \theta_2, \ldots, \theta_\ell\}$ a homogeneous
$S$-basis of $D(\CA)$.
\end{proof}

Note that the proof of 
Proposition \ref{prop:reversedelta} shows that $m_0$ is 
necessarily an exponent of $(\CA, \delta)$ if the latter is free
(so it does no harm to require this in the hypothesis).

The following is a special case of a general 
compatibility result 
of the Euler multiplicity with localizations from 
\cite[Lem.~2.14]{hogeroehrleschauenburg:free}.

\begin{lemma}
\label{lem:euler}
Let $X \in L(\CA)$, $H_0 \in \CA_X$,
and let $\CA''$ be the restriction 
with respect to $H_0$. Let $\delta = \delta_{H_0,m_0}$ 
be as in Definition \ref{def:delta}. 
Then we have 
\begin{itemize}
\item[(i)]
$((\CA_X)'', (\delta_X)^*) = ((\CA'')_X, (\delta^*)_X)$, and 
\item[(ii)]
$((\CA_X)'', \kappa) = ((\CA'')_X, \kappa_X)$ 
(where $\kappa$ on the left is the canonical multiplicity 
resulting from restriction of $\CA_X$ to $H_0$).
\end{itemize}
\end{lemma}

\begin{proof}
Part (i) follows from 
\emph{loc.~cit}.
Part (ii) follows from part (i) together with 
the fact that 
$((\CA_X)'', (\delta_X)^*) = ((\CA_X)'', \kappa)$, 
cf.~Proposition \ref{prop:delta}(ii).
\end{proof}

\subsection{Inductive Freeness for Multiarrangements}
\label{ssec:indutive}
As in the simple case, Theorem \ref{thm:add-del} motivates 
the notion of inductive freeness. 

\begin{defn}[{\cite[Def.~0.9]{abeteraowakefield:euler}}]
\label{def:indfree}
The class $\CIFM$ of \emph{inductively free} multiarrangements 
is the smallest class of multiarrangements subject to
\begin{itemize}
\item[(i)] $\Phi_\ell \in \CIFM$ for each $\ell \ge 0$;
\item[(ii)] for a multiarrangement $(\CA, \mu)$, if there exists a hyperplane $H_0 \in \CA$ such that both
$(\CA', \mu')$ and $(\CA'', \mu^*)$ belong to $\CIFM$, and $\exp (\CA'', \mu^*) \subseteq \exp (\CA', \mu')$, 
then $(\CA, \mu)$ also belongs to $\CIFM$.
\end{itemize}
\end{defn}

\begin{remark}
\label{rem:rank2indfree}
As for simple arrangements, if $r(\CA) \le 2$,
then $(\CA, \mu)$  is inductively free,  
\cite[Cor.~7]{ziegler:multiarrangements}.
\end{remark}

\begin{remark}
\label{rem:indtable}
In analogy to the simple case, 
cf.~\cite[\S 4.3, p.~119]{orlikterao:arrangements}, 
\cite[Rem.~2.9]{hogeroehrle:indfree}, 
it is possible to describe an 
inductively free multiarrangement $(\CA, \mu)$
by means of a so called 
\emph{induction table}.
In this process we start with an inductively free multiarrangement
(frequently $\Phi_\ell$) and add hyperplanes successively ensuring that 
part (ii) of Definition \ref{def:indfree} is satisfied.
We refer to this process as \emph{induction of hyperplanes}.
This procedure amounts to 
choosing an order on consecutive multiplicities $\mu_i$ of 
the simple arrangement  
$\CA$ 
for $i = 1, \ldots, n = |\mu|$ 
(see Definition \ref{def:order}), so that $|\mu_i | = i$, $\mu_n = \mu$,
and each restriction
$(\CA'', \mu_i^*)$ 
is inductively free.
As in the simple case, 
in the associated induction table 
we record in the $i$-th row the information 
of the $i$-th step of this process, by 
listing $\exp (\CA',\mu_i') = \exp (\CA,\mu_{i-1})$, 
the defining form $\alpha_{H_i}$ of $H_i$, 
as well as $\exp (\CA'', \mu_i^*)$, 
For instance, see Tables 
\ref{indtableD}, 
\ref{indtableGrr4}, 
and \ref{indtableG29} below. 
\end{remark}

We record an easy consequence
of Proposition \ref{prop:delta}(ii) in the context of 
inductive freeness which is going to be very useful in 
our proof of Theorem \ref{thm:Wdelta}.

\begin{corollary}
\label{cor:indfreedelta}
Suppose that both
$\CA$ and $(\CA'', \kappa)$ 
are inductively free.
Then  
$(\CA, \delta)$ is inductively free.
\end{corollary}

\begin{proof}
The result follows readily from induction on $m_0 \ge 1$,
Proposition \ref{prop:delta}(ii) and a repeated 
application of the addition part of Theorem \ref{thm:add-del}.
\end{proof}

The following special case of Corollary \ref{cor:indfreedelta} 
is immediate from \cite[Cor.~7]{ziegler:multiarrangements}.

\begin{corollary}
\label{cor:indfreedelta3}
Suppose that the rank 3 arrangement
$\CA$ is inductively free.
Then $(\CA, \delta)$ is inductively free.
\end{corollary}

The following is is a very useful tool for showing that a given 
multiarrangement is 
not inductively free by exhibiting a localization which fails to be  
inductively free.

\begin{theorem}
[{\cite[Thm.~1.3]{hogeroehrleschauenburg:free}}]
\label{thm:localmulti}
The class
$\CIFM$ is
closed under taking localizations.
\end{theorem}

We also require the fact that 
inductive freeness for 
multiarrangements behave well with the 
product construction. 

\begin{theorem}
[{\cite[Thm.~1.4]{hogeroehrleschauenburg:free}}]
\label{thm:products}
A product of multiarrangements belongs to 
$\CIFM$ 
if and only if 
each factor belongs to $\CIFM$.
\end{theorem}

We close this section with an example which shows that 
the reverse implication of Corollary \ref{cor:kappa2}
fails in general.

\begin{example}
\label{ex:kappa2}
Consider the complex $4$-arrangement $\CA$ given 
by the defining polynomial
$Q(\CA) := x y z t (x-y)(x-z)(x-y+t)(x-z+t)$.
Then one checks that  
 $\CA^{H_t}$ is inductively free,
 while both $\CA$ and $(\CA^{H_t},\kappa)$
are not free.
\end{example}

\section{Filtrations of Free Multiplicities}
\label{sec:filtrations}

In the sequel, we denote by $\one$ the simple multiplicity on a given arrangement. 
Recall the partial order on the set of multiplicities on an arrangement from 
Definition \ref{def:order}.

\begin{defn}
\label{def:freefiltration}
Let $\CA$ be a free arrangement. 
Suppose there is a free multiplicity $\mu > \one$ on $\CA$ 
such that 
there is a sequence of free multiplicities $\mu_i$ on $\CA$ satisfying
$\mu_i < \mu_{i+1}$ and $|\mu_{i+1}| = |\mu_i| +1$,
for $i = 1, \ldots, n-1$, where $\mu_1 := {\one}$ and $\mu_n := \mu$.
Then we say that the sequence of multiarrangements 
$(\CA, \mu_i)$ is a \emph{free filtration} of $(\CA, \mu)$ or
simply that the sequence 
$\mu_i$ is a \emph{filtration} of free multiplicities on $\CA$.
\end{defn}

\begin{lemma}
\label{lem:indfreechain}
Let $\CA$ be an inductively free arrangement. 
Suppose there is a free multiplicity $\mu > \one$ on $\CA$ 
along with a free filtration  
$\one = :\mu_1 < \ldots < \mu_n := \mu$ 
of $(\CA, \mu)$ as in Definition \ref{def:freefiltration}.
If each restriction  
along the free filtration is inductively free, then 
so is $(\CA, \mu)$.
\end{lemma}

\begin{proof}
It follows from Theorem \ref{thm:restriction} 
that the exponents of $(\CA', \mu_i') = (\CA, \mu_{i-1})$ and of $(\CA, \mu_i)$
differ in precisely one entry by $1$.  
So that, by Theorem \ref{thm:add-del}, 
the exponents of each restriction along the free filtration satisfy 
$\exp(\CA'', \mu_i^*) \subseteq \exp (\CA', \mu_i')$.
Now, since $\CA$ as well as each restriction 
$(\CA'', \mu_i^*)$ is inductively free,
a repeated 
application of the addition part of Theorem \ref{thm:add-del}
gives the desired result.
\end{proof}

We record a special case of Lemma \ref{lem:indfreechain}.

\begin{corollary}
\label{cor:indfreechain}
Let $\CA$ be an inductively free $3$-arrangement. 
Suppose there is a free multiplicity $\mu \ge \one$ on $\CA$ 
along with a free filtration as in Lemma \ref{lem:indfreechain}.
Then $(\CA, \mu)$ is inductively free.
\end{corollary}

\begin{proof}
Since each restriction along the free filtration is 
of rank $2$, it is already inductively free, by Remark \ref{rem:rank2indfree}.
So the result is immediate from Corollary \ref{cor:indfreechain}.
\end{proof}

In our next result we present a mild condition on a free multiplicity $\mu$ 
of a free arrangement $\CA$ which implies that every intermediate 
multiplicity $\one < \nu < \mu$ is also free.

\begin{lemma} 
\label{lem:multifreetube}
Let $\CA$ be a free arrangement with exponents $1 \le e_2 \le \ldots \le e_\ell$.
Assume that there is a free multiplicity $\mu > \one$ on $\CA$ with
$\exp(\CA,\mu) = \{e, e_2,\ldots,e_\ell\}$, where   
$e = 1 + \vert \mu \vert - \vert \CA \vert$.
Suppose that $\vert \mu \vert - \vert \CA \vert \ge e_\ell$. 
Let $\nu$ be a multiplicity satisfying $\one < \nu <  \mu$.
Then $(\CA,\nu)$ is free with 
$\exp(\CA,\nu) = \{\tilde e , e_2, \ldots, e_\ell\}$, where
$\tilde e = 1+ \vert \nu \vert - \vert \CA \vert$.

Moreover, if $\theta_1,\ldots, \theta_\ell$ 
is a homogeneous basis of $D(\CA,\mu)$ with $\pdeg \theta_i = e_i$ for 
$i=2,\ldots,\ell$ and $\pdeg \theta_1 = e$, then
\begin{equation}
\label{eq:basis}
\left(\prod_{H \in \CA}\alpha_H^{\nu(H)-1}\right) \theta_E,\theta_2,\ldots,\theta_\ell
\end{equation}
is a basis of $D(\CA,\nu)$.
\end{lemma}

\begin{proof}
Let $\theta_1,\ldots,\theta_\ell$ be a homogeneous basis of $D(\CA,\mu)$ 
with $\pdeg \theta_i = e_i$ for 
$i=2,\ldots,\ell$ and 
$\pdeg \theta_1 = e = 1+\vert \mu \vert - \vert \CA \vert$. 
It is enough to show that $\theta_E,\theta_2,\ldots,\theta_\ell$ 
is a basis of $D(\CA)$. 
For, then the derivations given
in \eqref{eq:basis} are linearly 
independent over $S$ and obviously also members 
of $D(\CA,\nu)$ 
(as $\nu \le \mu$) so that
the sum over their polynomial degrees equals $\vert \nu \vert$,
as $\sum_{i=2}^\ell e_i = |\CA| -1$.
We are going to check that the conditions of 
\cite[Thm.~4.42]{orlikterao:arrangements} hold. For that we have to show that 
\begin{equation}
\label{eq:theta1}
\theta_i \not \in S \theta_E + S\theta_2 + \cdots + S \theta_{i-1}
\end{equation}
holds for $i=2,\ldots,\ell$. 
Assume that \eqref{eq:theta1} fails for some $i$. Then there are 
polynomials $p,p_2,\ldots,p_{i-1} \in S$ such that
\begin{equation}
\label{eq:theta2}
p \, \theta_E = \sum_{k=2}^{i-1} p_k \theta_k + \theta_i.
\end{equation}
Consequently, since $\theta_2,\ldots,\theta_i$ are 
linearly independent over $S$, 
$p \not = 0$.
Now applying \eqref{eq:theta2} to  $\alpha_H$
for $H \in \CA$, we get that 
$ p \cdot \alpha_H \in \alpha_H^{\mu(H)} S$. Hence
\[
  p \in \left(\prod_{H \in \CA} \alpha_H^{\mu(H)-1}\right) S.
\]
Therefore, 
$\deg p \ge \sum_{H \in \CA} (\mu(H)-1) = \vert \mu \vert - \vert \CA \vert$. 
By \eqref{eq:theta2}, we have 
$1 + \deg p = \pdeg (p \, \theta_E) = \pdeg \theta_i$.
Using the hypothesis $\vert \mu \vert - \vert \CA \vert \ge e_\ell$,
we get
\[
\pdeg \theta_i \ge 1 + \vert \mu \vert - \vert \CA \vert 
  > \vert \mu \vert - \vert \CA \vert \ge e_\ell 
= \pdeg \theta_\ell \ge \pdeg \theta_i,
\] 
which is absurd.
Consequently, \eqref{eq:theta1} holds for all $i =2,\ldots,\ell$ and so 
$\theta_E,\theta_2,\ldots,\theta_\ell$ is a basis of $D(\CA)$, as claimed, by
\emph{loc.~cit.}
\end{proof}

We record an important consequence of Lemma \ref{lem:multifreetube}
and Theorem \ref{thm:add-del} which shows that 
the multisets of exponents of the restrictions 
along a free filtration, as in Lemma \ref{lem:multifreetube}, 
are constant and do not depend on 
the multiplicities $\mu^*_i$.

\begin{corollary}
\label{cor:multifreetuberestrictions}
Let $\CA$ and $(\CA, \mu)$ be as in Lemma \ref{lem:multifreetube}.
Let $\one = :\mu_1 < \ldots < \mu_n := \mu$ be 
a free filtration of $(\CA, \mu)$ as in Definition \ref{def:freefiltration}.
Then for each restriction, we have $(\CA'', \mu^*_i) = (\CA'',\kappa)$,
where $\kappa$ is Ziegler's canonical multiplicity on $\CA''$.
In particular, each such restriction along the filtration is free with 
$\exp (\CA'', \mu^*_i) = \exp(\CA'',\kappa) = \{e_2, \ldots, e_\ell\}$, 
where $\exp \CA =\{1, e_2, \ldots, e_\ell\}$.
\end{corollary}
 
\begin{proof}
Let $H_0 \in \CA$ and let $\one \le \nu' < \nu \le \mu$ be free multiplicities as given by
Lemma \ref{lem:multifreetube}, where $(\CA', \nu')$ is the deletion of $(\CA, \nu)$ with 
respect to $H_0$.
Let $\theta_1,\theta_2,\ldots,\theta_\ell$ be a free basis of $D(\CA', \nu')$.
It follows from 
Lemma \ref{lem:multifreetube} that 
$\alpha_0\theta_1,\theta_2,\ldots,\theta_\ell$ is a free basis of $D(\CA, \nu)$, as given by 
\eqref{eq:basis},
where $\ker \alpha_0 = H_0$. 
Owing to 
Theorem \ref{thm:restriction2}, we see that $\overline \theta_2 ,\ldots,\overline \theta_\ell$ is a basis
of $D(\CA'',\nu^*)$. 

Now let $\delta$ be the concentrated multiplicity on $\CA$ given by $\delta(H_0) = 2$.
Then $\one < \delta \le \mu$. Hence, according to Lemma \ref{lem:multifreetube}, 
$(\CA, \delta)$ is free with
 $\alpha_0\theta_E,\theta_2,\ldots,\theta_\ell$ being a basis of $D(\CA, \delta)$,
by \eqref{eq:basis}.
Thanks to Proposition \ref{prop:delta}, $(\CA'', \delta^*) = (\CA'', \kappa)$.

Since $\theta_E,\theta_2,\ldots,\theta_\ell$ is a basis of $D(\CA)$,
it follows again from Theorem \ref{thm:restriction2}  
applied to $(\CA, \delta)$ and $(\CA', \delta') = (\CA, \one)$,
that $\overline\theta_2,\ldots,\overline\theta_\ell$ is also a basis of 
$D(\CA'', \kappa)$. Hence, thanks to \eqref{eq:ziegler-saito}, $\nu^* = \kappa$.
\end{proof}

Next we apply Lemma \ref{lem:multifreetube} in the context of 
Ziegler's canonical multiplicity $\kappa$.

\begin{corollary}
\label{cor:multifreetube}
Let $\CA$ be a free arrangement with exponents 
$1 = e_1 \le e_2 \le \cdots \le e_{\ell-1} < e_\ell$ and
let $H \in \CA$ such that $\CA^H$ is free with 
$\exp(\CA^H) = \{1,e_2,\ldots,e_{\ell-1}\}$. 
Then $(\CA^H,\mu)$ is free for every multiplicity
$\one < \mu < \kappa$ with 
$\exp(\CA^H,\mu) = \{1+\vert \mu \vert-\vert \CA^H \vert,e_2,\ldots,e_{\ell-1}\}$.
\end{corollary}

\begin{proof}
Since $\CA^H$ is free with $\exp(\CA^H) = \{1,e_2,\ldots,e_{\ell-1}\}$,
we have $\vert \CA^H \vert  = 1+e_2+\ldots+e_{\ell-1}$.
Also note that $\vert \kappa \vert = \vert \CA \vert -1$.
Therefore, 
\[
1 + \vert \kappa \vert - \vert \CA^H \vert 
= \vert \CA \vert - (1+e_2+\ldots+e_{\ell-1}) = e_{\ell}.
\]
Consequently, 
$\exp(\CA^H,\kappa) = \{e_\ell, e_2,\ldots,e_{\ell-1}\} = 
\{1+ \vert \kappa \vert - \vert \CA^H \vert , e_2,\ldots,e_{\ell-1}\}$,
by Theorem \ref{thm:zieglermulti}.
Thus $\vert \kappa \vert - \vert \CA^H \vert  = e_{\ell} - 1 \ge e_{\ell-1}$, 
by hypothesis. So the result follows from 
Lemma \ref{lem:multifreetube} applied to $\CA^H$ and $\kappa$. 
\end{proof}

Corollary \ref{cor:multifreetube} yields that 
under rather weak assumptions,
\emph{every} multiplicity between $\one$ and the Ziegler multiplicity
$\kappa$  is free.
This yields an abundance of free filtrations from $\one$ to 
$\kappa$ on $\CA^H$. 
We demonstrate this in the context of reflection arrangements. 

\begin{proposition}
\label{prop:multifreetube}
Let $W$ be  of type $A_{\ell}$, $B_{\ell}$, $G(r,1,\ell)$,  
$F_4$, $H_4$, $G_{31}$, $G_{32}$ $E_6$, $E_7$, or $E_8$.
Let $\CA = \CA(W)$. 
Then for any $H \in \CA$ and every 
multiplicity $\one \le \mu \le \kappa$,
the multiarrangement 
$(\CA^H,\mu)$ is free with 
$\exp(\CA^H,\mu) 
= \{1+\vert \mu \vert - \vert \CA^H \vert,e_2,\ldots,e_{\ell-1}\}$.
\end{proposition}

\begin{proof}
It readily follows from
\cite[Props.~6.73, 6.77]{orlikterao:arrangements} and 
the tables in 
\cite[App.~C]{orlikterao:arrangements} that in the given cases any restriction
$\CA^H$ satisfies the condition of 
Corollary \ref{cor:multifreetube}.
\end{proof}

Armed with 
Corollaries \ref{cor:multifreetuberestrictions} and \ref{cor:multifreetube}, 
we can now prove our first main result of our classification.

\begin{proposition}
\label{prop:monomial}
Let $\CA = \CA(W)$ be the reflection arrangement of 
the full monomial group 
$W = G(r,1, \ell)$ for $r, \ell \ge 2$.
Then both 
$(\CA'', \kappa)$  and $(\CA, \delta)$ 
are inductively free. 
\end{proposition}

\begin{proof}
Owing to Theorem \ref{thm:indfree1} and 
Corollary \ref{cor:indfreedelta}, it is enough to show that 
$(\CA'',\kappa)$ is inductively free. 
We prove the latter by induction on $\ell$.
Clearly, for $\ell \le 3$, the statement holds. Thus suppose that 
$\ell \ge 4$ and that the result is true for smaller ranks. 

Note that 
irrespective of the choice of hyperplane in $\CA$, we have 
\begin{equation}
\label{eq:restr}
\CA'' \cong \CA(G(r,1, \ell-1)), 
\end{equation}
by \cite[Prop.\ 2.11]{orliksolomon:unitaryreflectiongroups}
(cf.~\cite[Prop.~6.82]{orlikterao:arrangements}).
Owing to \cite[Prop.\ 2.13]{orliksolomon:unitaryreflectiongroups}
(cf.~\cite[Prop.~6.77]{orlikterao:arrangements}), we have
\begin{equation}
\label{eq:exp}
\exp \CA = \{1, r+1, \ldots, (\ell-1)r +1\}. 
\end{equation}

Thanks to Proposition \ref{prop:multifreetube},
each multiplicity $\mu$ on $\CA''$ with
$\one \le \mu \le \kappa$ is 
free. So we pick an arbitrary free filtration 
of $\CA''$ from $\one$ to $\kappa$. 
Since $|\kappa| - |\CA''| = (\ell-1) r > (\ell-2) r + 1$,
which is 
the highest exponent of $\CA''$, by 
\eqref{eq:restr} and \eqref{eq:exp},
the hypotheses of Lemma \ref{lem:multifreetube} are satisfied for 
$\CB := \CA''$.
Consequently, Corollary \ref{cor:multifreetuberestrictions} 
implies that for each multiplicity $\mu$ along our free chain
we have $(\CB'',\mu^*) = (\CB'',\kappa_1)$, 
where $\kappa_1$ denotes the Ziegler multiplicity on $\CB''$.
Since $\CB'' \cong \CA(G(r,1, \ell-2))$,
by \eqref{eq:restr}, 
it follows from our induction hypothesis that $(\CB'',\kappa_1)$ is 
inductively free.
Consequently, since 
the simple arrangement $\CA''$ is inductively free,
thanks to Theorem \ref{thm:indfree1}, 
it follows that $(\CA'',\kappa)$ is inductively free,
by a repeated application of 
the addition part of 
Theorem \ref{thm:add-del}.
\end{proof}

Moreover, 
Corollaries \ref{cor:indfreechain} and \ref{cor:multifreetube}, 
imply the following consequence of Theorem \ref{thm:indfree2}.

\begin{corollary}
\label{cor:indfreerank3}
Let $W$ be of type $F_4$, $H_4$, $G_{31}$ or $G_{32}$.
Then $(\CA(W)'',\kappa)$ is inductively free.
\end{corollary}

\begin{proof}
By Proposition \ref{prop:multifreetube}, in each of the given cases 
the condition in Corollary \ref{cor:multifreetube}
on the exponents of $\CA(W)''$ is satisfied, so that every chain of 
multiplicities between $\one$ and $\kappa$  on $\CA(W)''$ is free.
By Theorem \ref{thm:indfree2}, 
$\CA(W)''$ is inductively free in each instance.
Therefore, it follows from Corollary \ref{cor:indfreechain}
that also $(\CA(W)'',\kappa)$ is inductively free.
\end{proof}

\section{Proofs of Theorems \ref{thm:kappa1} and \ref{thm:Wdelta}}
\label{sec:proofs}

Throughout, 
$W$ denotes a complex reflection group.
Owing to 
\cite[Prop.~2.10]{hogeroehrle:indfree}
and Theorem \ref{thm:products},
we may assume that $W$ is irreducible.
Fix $H_0 \in \CA$ and let $\CA''$  be the
restriction of $\CA$ with respect to $H_0$.
Moreover, fix $m_0 \in \BBZ_{\ge 1}$
and let $\delta = \delta_{H_0,m_0}$
as in Definition \ref{def:delta}.
We prove Theorems \ref{thm:kappa1} and \ref{thm:Wdelta}
more less simultaneously making repeated 
use of Corollary \ref{cor:indfreedelta}.

\subsection{Groups of low rank}
\label{ssect:rank2}
It is immediate from 
\cite[Cor.~7]{ziegler:multiarrangements}
that the reverse implication 
in Theorem \ref{thm:kappa1}
holds provided 
$W$ has rank at most $3$.
Likewise Theorem \ref{thm:Wdelta} 
holds for  
$W$ of rank at most 2.

\subsection{Coxeter Groups}
\label{ssect:coxeter}
Let $W$ be an irreducible Coxeter group.
By \cite{cuntz:indfree}, every Coxeter arrangement 
is inductively free.
So once we have shown that 
$(\CA(W)'', \kappa)$ is inductively free, so is 
$(\CA(W), \delta)$, thanks to 
Corollary \ref{cor:indfreedelta}.

\begin{lemma}
\label{lem:typea}
Let $\CA = \CA(A_{\ell-1})$ be the Coxeter arrangement of type $A_{\ell-1}$.
Then both $(\CA'', \kappa)$
and $(\CA, \delta)$
are inductively free.
\end{lemma}

\begin{proof}
For $\ell \le 4$, the result follows from \S \ref{ssect:rank2} and 
Corollary \ref{cor:indfreedelta3}.
So suppose that $\ell \ge 5$.
Since the underlying Coxeter group 
is transitive on $\CA$, there is no harm in fixing $H_0 = \ker (x_1 - x_2)$.
Thanks to \cite[Prop.~6.73]{orlikterao:arrangements},
$\CA'' \cong \CA(A_{\ell-2})$.
One readily checks that Ziegler's
multiarrangement $(\CA(A_{\ell-2}), \kappa)$ has defining 
polynomial
\[
Q(\CA(A_{\ell-2}), \kappa) = \prod_{2 \le j \le \ell-2} (x_1 - x_j)^2  \prod_{2 \le i<  j \le \ell-2} (x_i - x_j).
\]
It follows from \cite[Cor.~1.2]{conradroehrle:indfree}
that $(\CA(A_{\ell-2}), \kappa)$ is inductively free.
Since $\CA$ is inductively free, it follows from 
Corollary \ref{cor:indfreedelta}
that $(\CA, \delta)$ is inductively free.
\end{proof}

The following is the special case $r = 2$ in Proposition \ref{prop:monomial}.

\begin{lemma}
\label{lem:typeb}
Let $\CA = \CA(B_\ell)$ be the Coxeter arrangement of type $B_\ell$.
Then both $(\CA'', \kappa)$
and $(\CA, \delta)$
are inductively free.
\end{lemma}

\begin{lemma}
\label{lem:typed}
Let $\CA = \CA(D_\ell)$ be the Coxeter arrangement of type $D_\ell$
for $\ell \ge 3$.
Then both $(\CA'', \kappa)$
and $(\CA, \delta)$
are inductively free.
\end{lemma}

\begin{proof}
We first prove that 
$(\CA'', \kappa)$ is inductively free.
We argue by induction on $\ell \ge 3$.
For $\ell = 3$, we have $\CA(D_3) = \CA(A_3)$ and
so then the result follows from Lemma \ref{lem:typea}.

Now let $W = W(D_\ell)$ for $\ell \ge 4$
and suppose that the result holds for all instances of 
smaller rank.
Since $W$ is transitive on $\CA$, we may 
fix $H_0 = \ker(x_1 - x_2)$ without loss.
It follows from 
\cite[Prop.~6.85]{orlikterao:arrangements}
that 
$\CA'' = \CA^1_{\ell-1}(2)$.

One readily checks that 
Ziegler's
multiarrangement $(\CA(D_\ell)'', \kappa)$ has defining 
polynomial
\[
Q(\CA(D_{\ell})'', \kappa) = Q(\CA^1_{\ell-1}(2), \kappa) 
= x_1  \prod_{2 \le j \le \ell-1} (x_1^2 - x_j^2)^2
\prod_{3 \le i<  j \le \ell-1} (x_i^2 - x_j^2).
\]
According to Theorem \ref{thm:zieglermulti},
$(\CA(D_\ell)'',\kappa)$ is free with 
\begin{equation}
\label{eq:kappaD}
\exp (\CA(D_\ell)'',\kappa) 
= \exp (\CA^1_{\ell-1}(2) , \kappa)
= \{3,5, \ldots, 2(\ell-3)+1, \ell-1\}.
\end{equation}
By our induction hypothesis and  
Theorem \ref{thm:products},
the sub-multiarrangement
$(\CA(D_{\ell-1})'', \kappa) \times \Phi_1$ of 
$(\CA(D_\ell)'', \kappa)$ is inductively free.
By \eqref{eq:kappaD}, its set of exponents is
\[
\exp \left((\CA(D_{\ell-1})'', \kappa) \times \Phi_1\right)
= \exp \left((\CA^1_{\ell-2}(2) , \kappa) \times \Phi_1\right)
= \{3,5, \ldots, 2(\ell-4)+1, \ell-2, 0\}.
\] 
This forms the start of our induction table. 

We first add and restrict to the hyperplanes 
$\ker (x_1 - x_{\ell-1})$ and 
$\ker (x_1 + x_{\ell-1})$ twice.
Here we use Proposition 
\ref{prop:Euler} in order to determine the 
Euler multiplicity on the restrictions.
We consider the special case when $\ker (x_1 + x_{\ell-1})$ is 
added and restricted to for the second time.
Then in the notation of Proposition \ref{prop:Euler}, 
when considering 
the localization
$\CA_X = \{\ker (x_1 + x_{\ell-1}), \ker (x_1 - x_{\ell-1}), \ker x_1\}$,
we have $|\CA_X| = 3$, $|\mu_X| = 5$ and so 
Proposition 
\ref{prop:Euler}(2) applies.

Note that each of the restrictions is 
isomorphic to 
$(\CA^1_{\ell-2}(2) , \kappa) $
which is inductively free, by our inductive hypothesis.
Consequently, since in each case both $(\CA', \mu')$ and 
$(\CA'', \mu^*)$ are inductively free and 
$\exp(\CA'', \mu^*) \subseteq \exp(\CA', \mu')$, 
it follows that also $(\CA, \mu)$ is inductively free.

\begin{table}[ht!b]\small
\renewcommand{\arraystretch}{1.5}
\begin{tabular}{lll}
  \hline
  $\exp (\CA',\mu')$ & $\alpha_H$ & $\exp (\CA'', \mu^*)$ \\
  \hline
  \hline
$\exp \left((\CA^1_{\ell-2}(2) , \kappa) \times \Phi_1\right)$ & $x_1 - x_{\ell-1}$ &  
$\exp (\CA^1_{\ell-2}(2) , \kappa) $  \\
$= \{3,5, \ldots, 2(\ell-4)+1, \ell-2, 0\}$ & & \\
$\{3,5, \ldots, 2(\ell-4)+1, \ell-2, 1\}$ & $x_1 + x_{\ell-1}$ &  
$\exp (\CA^1_{\ell-2}(2) , \kappa) $   \\
$\{3,5, \ldots, 2(\ell-4)+1, \ell-2, 2\}$ & $x_1 - x_{\ell-1}$ &  
$\exp (\CA^1_{\ell-2}(2) , \kappa) $  \\
$\{3,5, \ldots, 2(\ell-4)+1, \ell-2, 3\}$ & $x_1 + x_{\ell-1}$ &  
$\exp (\CA^1_{\ell-2}(2) , \kappa) $  \\
$\{3,5, \ldots, 2(\ell-4)+1, \ell-2, 4\}$ & $x_2 - x_{\ell-1}$ &  
$\exp (\CA^1_{\ell-2}(2) , \kappa) $  \\
$\vdots$ & $\vdots$ & $\vdots$ \\
$\{3,5, \ldots, 2(\ell-4)+1, 2(\ell-3)+1, \ell-2\}$ & $x_{\ell-2} + x_{\ell-1}$ &  
$\exp (\CA^1_{\ell-2}(2) , \kappa) $  \\
$\exp (\CA(D_\ell)'',\kappa) = \{3,5, \ldots, 2(\ell-3)+1, \ell-1\}$ & & \\ 
\hline
\end{tabular}
\medskip
\caption{Induction Table for $(\CA(D_{\ell})'',\kappa)$.} 
\label{indtableD} 
\end{table}

It follows from Table \ref{indtableD} 
that $(\CA(D_{\ell})'', \kappa)$ is inductively free.
Since $\CA$ itself is inductively free, 
so is $(\CA, \delta)$, according to 
Corollary \ref{cor:indfreedelta}.
\end{proof}

\begin{lemma}
\label{lem:coxeterexceptional}
Let $W$ be an irreducible Coxeter group of exceptional type 
of rank at least $3$ and let
$\CA = \CA(W)$ be the Coxeter arrangement of $W$.
Then both  $(\CA'', \kappa)$
and $(\CA, \delta)$ are inductively free.
\end{lemma}

\begin{proof}
For $W$ of type $H_3$ the result follows from 
Corollary \ref{cor:indfreedelta3}.

For the remaining instances, 
we first prove that 
$(\CA'', \kappa)$
is inductively free.
For $W$ of type $F_4$ and $H_4$, this follows from 
Corollary \ref{cor:indfreerank3}.

Since $W(E_6)$ and $W(E_7)$ are parabolic subgroups 
of $W(E_8)$, we obtain the result for
the former from the latter, thanks to 
\cite[Cor.~6.28]{orlikterao:arrangements},
Lemma \ref{lem:euler}(ii), and Theorem \ref{thm:localmulti}. 

The case for $W = W(E_8)$ is considerably more complicated.
Since $W$ is transitive on $\CA$, we may take $H_0 = \ker x$.
It follows from Proposition \ref{prop:multifreetube} that 
$(\CA'',\nu)$ is free for every multiplicity 
$\nu$ with $\one \le \nu \le \kappa$.
So we pick a filtration 
$\one = \nu_1 < \nu_2 \ldots < \nu_{28} < \nu_{29} = \kappa$
of free multiplicities. 
Since $\CB := \CA''$ is inductively free, by \cite{cuntz:indfree}, 
it follows from Lemma \ref{lem:indfreechain} that 
$(\CA'', \kappa)$  is inductively free, once we have 
established
that the corresponding multirestrictions $(\CB'',\nu_i^*)$ 
are inductively free. 
Moreover, each multiplicity $\nu_i^*$ 
is just Ziegler's canonical multirestriction 
$\kappa_1$ stemming from the simple arrangement $\CB$,
by Corollary \ref{cor:multifreetuberestrictions}.
Let $\CC := \CB''$. 
Since $\CC$ is inductively free, by \cite{cuntz:indfree}, we 
can build an induction table by
starting with the 
simple arrangement $\CC$ and increasing the multiplicities
all the way to $(\CC, \kappa_1)$.
This is carried out in Table \ref{indtableE8}, where we 
indicate how to increase the 
multiplicities on the hyperplanes of the inductively free arrangement $\CC$
such that each resulting multiplicity is again a free multiplicity for $\CC$.
Note that in the 
first $16$ steps of the table the 
Euler multiplicity $\mu^*$ of the multirestriction is just Ziegler's
canonical restriction on the simple arrangement $\CC''$.
The remaining two Euler multiplicities are small 
extensions of Ziegler's canonical multiplicity.
In particular, here the order in which the hyperplanes 
are added plays a crucial role.
There are precisely $3$ different 
rank $5$ multiarrangements that occur as restrictions in 
Table \ref{indtableE8}.
We have checked that each of them is again inductively free, as required. 
However, we omit to include this data.
We also do not give linear forms for 
$\alpha_H$ in the table, 
but give the number of the induction step at the beginning of each row.

\begin{table}[ht!b]\small
\renewcommand{\arraystretch}{1.5}
\begin{tabular}{rlll}
  \hline
Step &  $\exp (\CA',\mu')$ & $\alpha_H$ & $\exp (\CA'', \mu^*)$\\
  \hline
  \hline
0 & $\exp (\CC) = \{1, 7, 11, 13, 14, 17\}$ 
& ... &  $\{7,11,13,14, 17\}$ \\
&$\vdots$ & $\vdots$ & $\vdots$ \\
15 & $\{7, 11, 13, 14, 16, 17\}$ 
& ... &  $\{7,11,13,14,17\}$ \\
16 &$\{7, 11, 13, 14, 17, 17\}$ 
& ... &  $\{7,11,13,17,17\}$ \\
& $\vdots$ & $\vdots$ & $\vdots$ \\
20 &$\{7, 11, 13, 17, 17, 18\}$ 
& ... &  $\{7,11,13, 17, 17\}$ \\
21 & $\{7, 11, 13, 17, 17, 19\}$ 
& ... &  $\{7,11,13,17,19\}$ \\
& $\vdots$ & $\vdots$ & $\vdots$ \\
26 & $\{7, 11, 13, 17, 19, 22\}$ 
& ... &  $\{7,11,13,17,19\}$ \\
& $\exp (\CC,\kappa_1) = \{7, 11, 13, 17, 19, 23\}$ \\ 
\hline
\end{tabular}
\medskip
\caption{Induction Table for $(\CC,\kappa_1)$.} 
\label{indtableE8} 
\end{table}

We mention that 
the only types of restrictions which do occur throughout 
this induction for $(\CA'', \kappa)$
are $(E_8,A_2)$ in rank $6$, $(E_8,A_3)$ in rank $5$, $(E_8,A_4)$ and
$(E_8,D_4)$ in rank $4$ and $(E_8,A_1 \times D_4)$, $(E_8,D_5)$ and $(E_8,A_5)$ in
rank $3$,  where we 
use the notation to indicate restrictions
following \cite[\S 3, App.]{orliksolomon:unitaryreflectiongroups} (cf. \cite[\S 6.4, App.~C]{orlikterao:arrangements}).

Finally, since $(\CA'', \kappa)$
is inductively free in each instance, 
it follows again from 
Corollary \ref{cor:indfreedelta}
that also 
$(\CA, \delta)$ is inductively free.
\end{proof}

\subsection{The Monomial Groups $G(r,1,\ell)$}
\label{ssect:monomial1}
For $W = G(r,1,\ell)$ for $r, \ell \ge 2$, the result 
follows from Proposition \ref{prop:monomial}.

\subsection{Non-real inductively free reflection arrangements of exceptional type}
\label{ssect:nonreal}
Let $W$ be irreducible of rank at least $3$ of exceptional type 
so that $\CA = \CA(W)$
is inductively free.
It then follows from Theorem \ref{thm:indfree1}
that $W$ equals $G_{25}$, $G_{26}$, or 
$G_{32}$.

\begin{lemma}
\label{lem:indfreeex}
Let $\CA = \CA(W)$ be the reflection arrangement 
where $W$ is  
$G_{25}$, $G_{26}$, or 
$G_{32}$.
Then both 
$(\CA'', \kappa)$ 
and 
$(\CA, \delta)$ 
are inductively free. 
\end{lemma}

\begin{proof}
Since $G_{25}$ and $G_{26}$
are both of rank $3$, it follows 
from Corollary \ref{cor:indfreedelta3}
that then 
$(\CA(W), \delta)$ is inductively free.
Since $\CA''$ is of rank $2$, also 
$(\CA'', \kappa)$ is inductively free.

For $W = G_{32}$ it follows from 
Corollary \ref{cor:indfreerank3} that 
$(\CA'', \kappa)$
is inductively free.
Owing to Corollary \ref{cor:indfreedelta},
also  $(\CA, \delta)$ is inductively free.
\end{proof}

\subsection{Non-inductively free reflection arrangements}
\label{ssect:nonindfree}
Suppose that $W$ is irreducible 
so that $\CA = \CA(W)$ is not inductively free.
Then, according to
Theorem \ref{thm:indfree1},  
either $W = G(r,r,\ell)$, for $r, \ell \ge 3$ or 
$W$ is of exceptional type 
$G_{24}$, $G_{27}$, $G_{29}$, $G_{31}$, $G_{33}$, or $G_{34}$.   

The following is
\cite[Prop.~4.1]{hogeroehrleschauenburg:free}, its
proof consists of an application of Theorem \ref{thm:localmulti}.

\begin{lemma}
\label{lem:grrl}
Let $\CA = \CA(W)$ be the reflection arrangement of 
the monomial group 
$W = G(r,r, \ell)$ for $r \ge 3$ and $\ell \ge 5$.
Then both 
$(\CA'', \kappa)$ and $(\CA, \delta)$ 
are not inductively free.
\end{lemma}

\begin{lemma}
\label{lem:grr3}
Let $W$ be $G(r,r,3)$, $G(r,r,4)$, for $r \ge 3$
or of exceptional type 
$G_{24}$, $G_{27}$, $G_{29}$, $G_{31}$, $G_{33}$, or $G_{34}$.   
Then $(\CA(W), \delta)$ 
is not inductively free.
\end{lemma}

\begin{proof}
Let $\CA = \CA(W)$ be the reflection arrangement for $W$
as in the statement. 
We argue by induction on $m_0 \ge 1$.
For $m_0 = 1$, the result follows from 
Theorem \ref{thm:indfree1}.
So suppose that $m_0 > 1$ and that the claim is proved 
for smaller values for $m_0$.
Pick $H \in \CA$ and let 
$(\CA, \CA', \CA'')$ be the triple
of $\CA$ with respect to $H$.
If $H = H_0$, then
$(\CA', \delta') = (\CA, \delta_{H_0,m_0-1})$
which is not inductively free by inductive hypothesis.

Now suppose that $H \ne H_0$.
In each case apart from $G_{31}$ it follows from 
\cite[Cor.~1.3, Cor.~2.4]{hogeroehrle:indfree}
that $\CA' = \CA\setminus \{H\}$ is not free for any choice of 
hyperplane.
Therefore, in each of these instances, 
since $\CA'$ is not free, neither is 
$(\CA', \delta')$, by Theorem \ref{thm:delta}.
In particular, 
$(\CA', \delta')$ is not inductively free.
Thus 
$(\CA', \delta')$ is not inductively free
for any choice of $H$ and 
therefore, neither is $(\CA, \delta)$.

For $W$ of exceptional type $G_{31}$, we
argue as follows. 
Suppose that $(\CA, \delta)$ is inductively free.
Let $\CA' = \CA \setminus \{H\}$ and 
$(\CA', \delta') = (\CA', \delta|_{\CA'})$.
If $(\CA', \delta')$ is free, then so is $\CA'$, again thanks to
Theorem \ref{thm:delta}. 
Continuing in this fashion removing further hyperplanes from 
the multiarrangement, 
we obtain a free chain 
of consecutive subarrangements of the simple 
arrangement $\CA(G_{31})$ 
ending at the empty $4$-arrangement.
But this contradicts the fact that there is 
no such free chain of $\CA(G_{31})$, according to the proof of 
\cite[Lem.~3.5]{hogeroehrle:indfree}.
\end{proof}

\begin{lemma}
\label{lem:grr4}
Let $W$ be $G(r,r,3)$, $G(r,r,4)$, for $r \ge 3$
or $W$ is exceptional of type 
$G_{29}$ or $G_{31}$.
Then $(\CA(W)'', \kappa)$ is inductively free.
\end{lemma}

\begin{proof}
For $W =G(r,r,3)$ the result is clear, thanks to Remark \ref{rem:rank2indfree}, 
and for $W = G_{31}$ the result follows from 
Corollary \ref{cor:indfreerank3}.

If $W = G(r,r,4)$ or $W$ is of type $G_{29}$, 
then $\CA(W)''$ is inductively free,
according to Theorem \ref{thm:indfree2}
($W$ is transitive on $\CA(W)$ 
in each of these cases). Consequently, we may 
initialize an induction table
in these instances with the simple inductively free arrangement $\CA(W)''$
and add hyperplanes  to reach the Ziegler multiplicity $\kappa$.

Let $W = G(r,r,4)$. Then, by
\cite[Prop.~6.82]{orlikterao:arrangements},
for any choice $H_0 \in \CA$, we have 
$\CA(W)'' \cong \CA_3^1(r)$ which is inductively free,
according to
Theorem \ref{thm:indfree2}(ii).
Fix $H_0 = \ker (x_1-x_4)$.
Thanks to 
\cite[Cor.~6.86]{orlikterao:arrangements},
we have 
\[
\exp \CA(G(r,r,4)) = \{1, r+1, 2r+1, 3(r-1)\}
\]
and therefore, by Theorem \ref{thm:zieglermulti}, 
$(\CA(G(r,r,4))'', \kappa)$ is free with 
\[
\exp (\CA(G(r,r,4))'', \kappa) = \{r+1, 2r+1, 3(r-1)\}.
\]
So in particular, $|\kappa| = 6r -1$.
Moreover, by
\cite[Prop.~6.85]{orlikterao:arrangements},
we have 
\[
\exp \CA(G(r,r,4))'' = \exp \CA_3^1(r) = \{1, r+1, 2r-1\}.
\]
The defining polynomial for 
$(\CA(G(r,r,4))'', \kappa)$ is given by 
\[
Q(\CA(G(r,r,4))'', \kappa) = Q(\CA_3^1(r), \kappa) 
= x_1^{r-1} \left(x_1^r - x_2^r \right)^2 \left(x_1^r - x_3^r \right)^2
\left(x_2^r - x_3^r \right).
\]
Table \ref{indtableGrr4} is the induction table for 
$(\CA(G(r,r,4))'', \kappa)$.
Note that since $\CA(G(r,r,4))''$ is of rank $3$,
each of the restrictions in Table \ref{indtableGrr4}
is of rank $2$ and so is inductively free.
Let $\zeta$ be a primitive $r\th$ root of unity.

There are two different types of steps in this induction table. 
The first one consists of adding a hyperplane 
$\ker(x_1- \zeta^i x_2)$ for some $i$ or 
$\ker(x_1- \zeta^j x_3)$ for some $j$ 
and the second one  
consists of adding $\ker x_1$.

In the first step, our defining polynomial is of the form
$$
Q(\CA,\mu) = x_1 (x_2^r-x_3^r) 
\prod_{0 \le i <r} (x_1 - \zeta^i x_2)^{k_i} (x_1-\zeta^i x_3)^{l_i}
$$
with $k_i,l_i \in \{1,2\}$ for all $i$.

We claim that $(\CA,\mu)$ is free with 
$\exp(\CA,\mu) = \{1+\sum_{0\le i <r}(k_i+l_i-2),r+1,2r-1\}$.
This is true if $k_i = l_i = 1$ for all $i$, 
since then $(\CA,\mu)$ is the simple arrangement $\CA_3^1(r)$. 
For each hyperplane in
this step the argument is essentially the same, 
so we only present it 
for $H = \ker(x_1- \zeta^i x_2)$. Fix 
$i$ and let $k_i = 1$.
The defining polynomial of the restriction is
$$
Q(\CA^H,\mu^*) = x_1^r \prod_{0\le j <r} (x_1-\zeta^j x_3)^2.
$$
The Euler multiplicities can be 
computed using Proposition \ref{prop:Euler}(2) and (3).
A basis of $D(\CA^H,\mu^*)$ is given by 
\begin{eqnarray*}
\theta_1 &=& r x_1^{r+1}D_1 + ((r+1)x_1^rx_3-x_3^{r+1})D_3,\\
\theta_2 &=& x_1^{r-1}(rx_1x_3^{r-1}D_1 + (x_1^r+(r-1)x_3^r))D_3.
\end{eqnarray*}
Therefore, we have  $\exp(\CA^H,\mu^*) = \{r+1,2r-1\}$. 
The Addition Deletion Theorem \ref{thm:add-del} then yields that
$(\CA,\mu)$ is free with the desired exponents.

In the second step we consider defining polynomials of the form
$$
Q(\CA,\mu) = x_1^i (x_1^r-x_2^r)^2(x_1^r-x_3^r)^2(x_2^r-x_3^r),
$$
for $1 \le i \le r-2$. 

We claim that $(\CA,\mu)$ is free with $\exp(\CA,\mu) = \{r+1,2r+1,2r-1+i\}$.
We know from the first step that this is true for $i=1$. 
Let $H = \ker x_1$. The restriction is given by 
\[
Q(\CA^H,\mu^*) = x_2^{r+1}x_3^{r+1}(x_2^r-x_3^r).
\]

Note that $(\CA^H,\mu^*)$ coincides with $(\CA(G(r,1,3))^{\ker x_1},\kappa)$.
Therefore, thanks to Theorem \ref{thm:zieglermulti}, 
$(\CA^H,\mu^*)$ is free and $\exp(\CA^H,\mu^*) = \{r+1,2r+1\}$.
Since the hypotheses of  
Theorem \ref{thm:add-del} are fulfilled, we get freeness of $(\CA,\mu)$ 
along with the desired exponents for each $i$. 

This completes the argument for $W = G(r,r,4)$.
We record the details of the induction in  
Table \ref{indtableGrr4}.

\begin{table}[ht!b]\small
\renewcommand{\arraystretch}{1.5}
\begin{tabular}{lll}
  \hline
  $\exp (\CA',\mu')$ & $\alpha_H$ & $\exp (\CA'', \mu^*)$\\
  \hline
  \hline
$\exp \CA_3^1(r) = \{1, r+1, 2r-1\}$ & $x_1 - \zeta x_2$ &  $\{r+1,2r-1\}$  \\
$\{2,r+1,2r-1\}$                     & $x_1 - \zeta^2 x_2$ &  $\{r+1,2r-1\}$  \\
$\vdots$ & $\vdots$                  & $\vdots$ \\
$\{r,r+1,2r-1\}$                     & $x_1 - x_2$ &  $\{r+1,2r-1\}$  \\
$\{r+1,r+1,2r-1\}$                   & $x_1 - \zeta x_3$ & $\{r+1,2r-1\}$ \\
$\vdots$ & $\vdots$                  & $\vdots$ \\
$\{r+1,2r,2r-1\}$                    & $x_1 - x_3$                    & $\{r+1,2r-1\}$ \\
$\{r+1,2r+1,2r-1\}$                  & $x_1$ & $\{r+1,2r+1\}$  \\
$\vdots$ & $\vdots$ & $\vdots$ \\
$\{r+1,2r+1,2r-1+(r-3)\}$ & $x_1$ &  $\{r+1,2r+1\}$  \\
$\exp (\CA(G(r,r,4))'',\kappa) = \{r+1, 2r+1, 3(r-1)\}$ \\ 
\hline
\end{tabular}
\medskip
\caption{Induction Table for $(\CA(G(r,r,4))'',\kappa)$.} 
\label{indtableGrr4} 
\end{table}

Next let $W$ be of type $G_{29}$.
Since $W$ is transitive on $\CA(G_{29})$, we may take $H_0 = \ker x_4$.
Then the defining polynomial for 
$(\CA(G_{29})'', \kappa)$ is given by 
\begin{align*}
Q(\CA(G_{29})'', \kappa) = \ & 
( x_{1} -  x_{2} - i x_{3} )^2 ( x_{1} -  x_{2} ) ( x_{2} -  x_{3} ) x_{3}^3 ( x_{1} - i x_{2} -  x_{3} )^2 ( x_{1} -  x_{2} + i x_{3} )^2\\
&  ( x_{1} -  x_{3} ) ( x_{1} + i x_{2} - i x_{3} )^2 x_{2}^3 ( x_{1} + i x_{2} -  x_{3} )^2 ( x_{1} - i x_{2} + i x_{3} )^2 \\
& x_{1}^3 ( x_{1} - i x_{2} + x_{3} )^2 ( x_{1} - i x_{2} - i x_{3} )^2 ( x_{1} + i x_{2} + i x_{3} )^2 ( x_{1} + i x_{2} + x_{3} )^2\\
&  ( x_{1} + x_{2} - i x_{3} )^2 ( x_{1} + x_{2} + i x_{3} )^2 ( x_{2} + x_{3} ) ( x_{1} + x_{3} ) ( x_{1} + x_{2} ).
\end{align*}

We present an induction table for 
$(\CA(G_{29})'', \kappa)$ in Table \ref{indtableG29}.
Note that since $\CA(G_{29})''$ is of rank $3$
each of the restrictions 
in Table \ref{indtableG29}
is of rank 2 and so is inductively free.

\begin{table}[ht!b]\small
\renewcommand{\arraystretch}{1.3}
\begin{tabular}{lll}
  \hline
  $\exp (\CA',\mu')$ & $\alpha_H$ & $\exp (\CA'', \mu^*), \mu^*$\\
  \hline
  \hline
$\exp \CA(G_{29})'' = \{1,9,11\}$ 
& $ x_{1} -  x_{2} - i x_{3} $ & $\{ 9,11 \},(1, 2, 2, 4, 2, 3, 4, 2) $ \\ 
 $\{2,9,11\}$ 
& $ x_{3} $ & $\{ 9,11 \},(3, 4, 3, 4, 3, 3) $ \\ 
 $\{3,9,11\}$ 
& $ x_{3} $ & $\{ 9,11 \},(3, 4, 3, 4, 3, 3) $ \\ 
 $\{4,9,11\}$ 
& $ x_{1} - i x_{2} -  x_{3} $ & $\{ 9,11 \},(2, 1, 2, 2, 4, 3, 4, 2) $ \\ 
 $\{5,9,11\}$ 
& $ x_{1} -  x_{2} + i x_{3} $ & $\{ 9,11 \},(1, 2, 2, 3, 2, 4, 4, 2) $ \\ 
 $\{6,9,11\}$ 
& $ x_{1} + i x_{2} - i x_{3} $ & $\{ 9,11 \},(2, 2, 3, 2, 1, 4, 2, 4) $ \\ 
 $\{7,9,11\}$ 
& $ x_{2} $ & $\{ 9,11 \},(3, 4, 3, 4, 3, 3) $ \\ 
 $\{8,9,11\}$ 
& $ x_{2} $ & $\{ 9,11 \},(3, 4, 3, 4, 3, 3) $ \\ 
 $\{9,9,11\}$ 
& $ x_{1} + i x_{2} -  x_{3} $ & $\{ 9,11 \},(2, 2, 2, 3, 2, 1, 4, 4) $ \\ 
 $\{9,10,11\}$ 
& $ x_{1} - i x_{2} + i x_{3} $ & $\{ 9,11 \},(2, 2, 3, 2, 1, 4, 2, 4) $ \\ 
 $\{9,11,11\}$ 
& $ x_{1} $ & $\{ 9,11 \},(3, 3, 3, 3, 4, 4) $ \\ 
 $\{9,11,12\}$ 
& $ x_{1} $ & $\{ 9,11 \},(3, 3, 3, 3, 4, 4) $ \\ 
 $\{9,11,13\}$ 
& $ x_{1} - i x_{2} + x_{3} $ & $\{ 9,13 \},(2, 2, 2, 5, 1, 2, 3, 5) $ \\ 
 $\{9,12,13\}$ 
& $ x_{1} - i x_{2} - i x_{3} $ & $\{ 9,13 \},(2, 1, 2, 5, 3, 5, 2, 2) $ \\ 
 $\{9,13,13\}$ 
& $ x_{1} + i x_{2} + i x_{3} $ & $\{ 9,13 \},(2, 1, 2, 5, 3, 5, 2, 2) $ \\ 
 $\{9,13,14\}$ 
& $ x_{1} + i x_{2} + x_{3} $ & $\{ 9,13 \},(2, 5, 2, 5, 2, 3, 1, 2) $ \\ 
 $\{9,13,15\}$ 
& $ x_{1} + x_{2} - i x_{3} $ & $\{ 9,13 \},(2, 2, 3, 1, 2, 5, 5, 2) $ \\ 
 $\{9,13,16\}$ 
& $ x_{1} + x_{2} + i x_{3} $ & $\{ 9,13 \},(1, 2, 2, 5, 2, 3, 5, 2) $ \\ 
$\exp (\CA(G_{29})'', \kappa) = \{9,13,17\}$ & & \\
\hline
\end{tabular}
\medskip
\caption{Induction Table for $(\CA(G_{29})'',\kappa)$.} 
\label{indtableG29} 
\end{table}

By the data in Tables \ref{indtableGrr4},
\ref{indtableG29} and Remark \ref{rem:indtable}, 
$(\CA'', \kappa)$
is inductively free in each instance.
\end{proof}

We mention that it follows from Theorem \ref{thm:indfree1} and 
Lemmas \ref{lem:grr3} and \ref{lem:grr4} that the implication in 
Corollary \ref{cor:indfreedelta}
fails without the assumption that $\CA''$ is inductively free.

\begin{lemma}
\label{lem:g33}
Let $W$ be of type 
$G_{33}$ or $G_{34}$.   
Then $(\CA(W)'', \kappa)$ is not inductively free.
\end{lemma}

\begin{proof}
First consider the case of $W = G_{33}$.
Let $\CA = \CA(W)$.
Direct computations have shown that there is no chain of free multiplicities 
from the zero multiplicity on $\CA''$ to $\kappa$. The maximal length of such 
chains is $22 < 44 = |\kappa|$. 
Hence $(\CA'',\kappa)$ is not inductively free.

Finally, let 
$W = G_{34}$ and $\CA = \CA(W)$. Then
$W$ admits a parabolic subgroup $W_X$ of type 
$G(3,3,5)$, by 
\cite[Table C.15]{orlikterao:arrangements}.
Thanks to \cite[Cor.~6.28]{orlikterao:arrangements},
we have $\CA(W_X) = \CA(W)_X = \CA_X$.
Fix $H_0 \in \CA_X$.
Then considering the restriction with respect to 
$H_0$, we have
$(\CA'')_X = (\CA_X)''$.
So we see that 
$((\CA'')_X, \kappa_X)  = ((\CA_X)'', \kappa_X) 
= (\CA(G(3,3,5))'', \kappa)$, by Lemma \ref{lem:euler}(ii).
According to Lemma \ref{lem:grrl}, the latter
is not inductively free, thus neither is 
 $(\CA(G_{34})'', \kappa)$, owing to 
Theorem \ref{thm:localmulti}.
\end{proof}

Finally, Theorems \ref{thm:kappa1} and \ref{thm:Wdelta}
follow from the results from \S \ref{ssect:rank2} 
to \S \ref{ssect:nonindfree} above.

\begin{remark}
\label{rem:computations}
The fact that 
 $(\CA(G_{33})'', \kappa)$ is not inductively free, proved in Lemma \ref{lem:g33},
and the data in the induction tables 
\ref{indtableE8}
and \ref{indtableG29}
were obtained by computational means.
Specifically, we 
retrieved 
relevant data about reflection groups
from some \GAP\ code provided by J.~Michel
\cite{michel:development}; see also \cite{gap3} and \cite{chevie}.
Subsequently, we made use of \Sage, cf.~\cite{sage}, 
and \Singular, cf.~\cite{singular},
in order to calculate
the various bases of modules of derivations involved. 
\end{remark}

\bigskip {\bf Acknowledgments}: 
We acknowledge 
support from the DFG-priority program SPP1489 
``Algorithmic and Experimental Methods in Algebra, Geometry, and Number Theory''.
Part of the research for this paper was carried out while the
authors were staying at the Mathematical Research Institute
Oberwolfach supported by the ``Research in Pairs'' programme.


\bigskip

\bibliographystyle{amsalpha}

\newcommand{\etalchar}[1]{$^{#1}$}
\providecommand{\bysame}{\leavevmode\hbox to3em{\hrulefill}\thinspace}
\providecommand{\MR}{\relax\ifhmode\unskip\space\fi MR }
\providecommand{\MRhref}[2]{%
  \href{http://www.ams.org/mathscinet-getitem?mr=#1}{#2} }
\providecommand{\href}[2]{#2}


\end{document}